\documentclass[11pt]{article}
\usepackage{amssymb,latexsym}
\usepackage{amsfonts}

\usepackage[a4paper,hypertex]{hyperref}

\usepackage{xypic}
\usepackage{tocbibind}

\newcommand{\ee}{\end{equation}}
\newcommand{\eea}{\end{eqnarray}}
\newcommand{\bean}{\begin{eqnarray*}}
\newcommand{\eean}{\end{eqnarray*}}
\newif\ifpctex
\newcommand{\noi}{\noindent}

\newcommand{\ve}{\varepsilon}
\newcommand{\wt}{\widetilde}

\newtheorem{theorem}{Theorem}
\newtheorem{remark}{Remark}
\newtheorem{example}{Example}

\newcommand{\qed}{{}\hfill {}
\hfill{$\square$}\vspace{0.3cm}\pagebreak[2]\par}
\newtheorem{proposition}{Proposition}[section]
\newtheorem{definition}[proposition]{Definition}
\newtheorem{lemma}[proposition]{Lemma}
\newtheorem{corollary}[proposition]{Corollary}

\newenvironment{proof}{\par\noindent{\bf Proof\ }}{\qed }

 \newcommand{\rand}[1]{}

\newcommand{\Rand}[1]{\marginpar{#1}}

\newcommand{\be}[1]{\Rand{\vspace{0,6cm}\tt #1}\begin{equation}\label{#1}}
\newcommand{\bea}[1]{\Rand{\vspace{0,7cm}\tt
#1\vspace{-0,7cm}}\begin{eqnarray}\label{#1}} \marginparwidth2.5cm

\newcommand{\beL}[1]{\Rand{\vspace{0,6cm}\tt #1}\begin{lemma}\label{#1}}
\newcommand{\beD}[1]{\Rand{\vspace{0,6cm}\tt #1}\begin{definition}\label{#1}}
\newcommand{\beT}[1]{\Rand{\vspace{0,6cm}\tt #1}\begin{theorem}\label{#1}}
\newcommand{\beP}[1]{\Rand{\vspace{0,6cm}\tt #1}\begin{proposition}\label{#1}}
\newcommand{\beC}[1]{\Rand{\vspace{0,6cm}\tt #1}\begin{corollary}\label{#1}}

\newcommand{\suml}{\sum\limits}

\newcommand{\N}{\mathbb{N}}

\newcommand{\I}{\mathbb{I}}

\begin{document}

\title{Multilevel mutation-selection systems and set-valued duals} \maketitle

\begin{center}
\author{{\bf Donald A. Dawson$^{(1)}$  }}
\end{center}


\begin{abstract} {A class of measure-valued processes which model multilevel multitype populations
undergoing mutation, selection, genetic drift and spatial migration is considered.   We investigate
the qualitative behaviour of models with  multilevel selection and  the interaction between the
different levels of selection.  The  basic tools in our analysis include the martingale problem
formulation for measure-valued processes and  a generalization of the function-valued and
set-valued dual representations introduced in Dawson-Greven (2014).  The dual is a powerful tool
for the analysis of the ergodic behaviour of these processes and  the study of evolutionary systems
which model phenomena including altruism, the emergence of cooperation and more complex
interactions.}

\end{abstract}

\vfill

 \noi {\bf Keywords:} Multilevel measure-valued process, multilevel selection, set-valued dual.\vspace*{0.5cm}

\noi {\bf AMS Subject Classification: 60J70, 92D25}
\medskip

\noindent \thanks{\bf This research is supported by NSERC}

  \vspace*{0.5cm}

\noi {\footnotesize $^{1)}$ {School of Mathematics and Statistics, Carleton University, Ottawa K1S
5B6, Canada\\ e-mail: ddawson@math.carleton.ca}}

\newpage

\tableofcontents

\newpage

\section{Introduction}

Multitype populations are naturally modeled as measure-valued processes.  In this paper we consider
a class of multilevel measure-valued processes which model ensembles  of subpopulations with
mutation and selection at the subpopulation level and possible death and replacement of
subpopulations. In particular this includes mathematical models of multilevel selection which has
been the subject of considerable debate in the evolutionary biology literature. Before introducing
our multilevel  models we begin with a brief review of some of this literature.

\subsection{Hierarchical population structure}

The hierarchical structure of populations plays a fundamental role in the biological and social
sciences. In evolutionary biology and ecology the hierarchy includes ecosystems, community,
species, organism, genes and in the social sciences we have cites, regions, nations, etc.  These
are systems in which at each level of the hierarchy we have a collection of elements of the next
lower level in the hierarchy.  The description of a unit at a given level in the hierarchy involves
the distribution of the different characteristics of the individuals at the next lower level in the
hierarchy.

\subsection{Historical remarks  on hierarchy in population genetics and evolutionary biology}

Biological evolution can be viewed in terms of a hierarchy of levels of organisation going from the
molecular level to the species level and social of groups of members of a species. A natural
question is to what extent does the Darwinian mechanism of variation and selection of fitter types
in a competitive environment play a role at the various levels.
\bigskip

 \noindent  An early application of group selection was by Wynne-Edwards (1962) \cite{WE62} who used it to explain adaptations and social behaviour of
 animals. Subsequently
 G.C. Williams (1966) \cite{W66}  made a highly critical analysis of  group selection which was very
 influential.
  John Maynard Smith
(1964),(1976)  (\cite{MS64},\cite{MS76}) considered both   group selection and kin selection which
was introduced by W.D. Hamilton (1964) \cite{H64} and concluded that there may be conditions under
which group selection is effective.   In the subsequent decades there has been intense debate among
evolutionary biologists about the extent to which evolution has been shaped by selective pressures
acting at the level of groups of individuals. \bigskip

\noindent  In recent years the role of multilevel selection has re-emerged in a number of contexts
including the emergence of life (Szathm\'ary and Demeter \cite{SD-87}), structural complexity
(G\"{o}rnerup and Crutchfield (2008) \cite{GP-06}),
 prebiotic evolution (Hogeweg and Takeuchi \cite{HT-03}), plasmid replication in bacteria (Paulsson (2002) \cite{P-02}), evolution of cooperation
 (Traulsen and Nowak \cite{TN-06}) and sociobiology (Wilson and Wilson) \cite{WW}.  In the study of cultural evolution Boyd and Richardson \cite{BR} suggest that
 interdemic group selection can be important when there are multiple stable equilibria at the deme level and the emergence
 of higher level equilibria occurs.
Moreover these ideas are relevant in the context of spatially structured populations and
evolutionary ecology (see Lion and van Baalen \cite{LVB-08}, Lion et al \cite{LJD-11}). A detailed
study of host-pathogen systems in the framework of multilevel selection was carried out by Luo
(2013) \cite{S-13}, \cite{SRMK-12} and Luo and Mattingly \cite{LM15} who demonstrated that a phase
transition between  the dominance of selection at the different levels can occur as the model
parameters are varied. Multilevel selection also underlies current research in the development of
complex human societies (see e.g. Turchin et al. \cite{Tur}).  Several books have been written on
the question of the levels of selection. These include Brandon and Burian \cite{BB-84}, Sober and
Wilson \cite{S-W}, Keller \cite{Kel-99} and Okasha \cite{O-06}.  An number of other recent research
papers on multilevel selection are included in the References.
\bigskip

\noi We end with a quotation of Leigh \cite{L10} that provides a useful perspective on these
questions:
 {\em ``These conditions ( e.g. he quotes Kimura's conditions - see (\ref{kcond})) seem so wonderfully improbable
that, following Williams (1966), most biologists have focused almost exclusively on individual
selection. Improbability, however, does not mean impossibility. Group selection capable of
overwhelming selection within groups, played a crucial role in some major transitions ...''}.

\bigskip

\noi An objective of this research is to develop tools to identify conditions under which higher
level selection is relevant for the class of mathematical models we consider.

\bigskip

\subsubsection{A multideme model with two types of individuals}

In order to introduce the main ideas we briefly review a formulation of group selection given by
Aoki \cite{A-82}). This begins with a countable collection of demes where each deme is a population
of $n$ individuals which are either type A or type B. Type B individuals are altruistic and add to
the fitness of the deme.  The life cycle of a deme involves four discrete events, namely,
migration, reproduction, extinction and recolonization. In the reproduction stage, within each deme
the population undergoes weighted finite population resampling in which type B has fitness $-s$
(with $s>0$). The probability that a deme suffers extinction is a monotone decreasing function of
the proportion of type B individuals it contains, that is, the fitness of the deme increases as a
function of the number of altruistic individuals it contains. In the migration stage a random
number of individuals within a deme are replaced by individuals chosen at random from the total
population pool.  Aoki then obtained a recursion formula for the probability distribution of the
number of individuals of type B per deme over successive life cycles  and discussed the question of
the long time behavior of this distribution, in particular whether or not the proportion of type B
goes to zero or not.

\subsubsection{A diffusion process model of multilevel selection with two types}

The class of Wright-Fisher diffusion processes plays an important role in population genetics.
Following Aoki \cite{A-82}), an analogous extension of the Wright-Fisher process was introduced by
Kimura (1983) \cite{K-83}  with  alleles $A$ and $B$ distributed in an infinite number of competing
demes. It is assumed that  $B$ is the altruistic allele which has a selective disadvantage $s_1$
but which is beneficial for a deme in competition with other demes,
 namely, a deme having frequency $x$ for $B$ has advantage $s_2(x-\int y\nu(dy))$ where $\nu(dy)$ is the distribution of the frequency of type $B$ individuals
  over the set of demes.
This leads to the integro-differential equation for the dynamics of the density $\{\wt\nu(t,x)
\}_{t\geq 0}$ where $\nu(t,dx)=\wt\nu(t,x)dx$ :

\be{}\label{Kim} \frac{\partial \wt\nu(t,x)}{\partial t}
=\frac{\gamma_1}{2}\frac{\partial^2}{\partial x^2}\left(
x(1-x)\wt\nu(t,x)\right)-\frac{\partial}{\partial x}(M(t,x)\wt\nu(t,x)) +s_2\left(x-\int
y\wt\nu(t,y)dy\right)\wt\nu(t,x)\ee where
\[ M(t,x)=m_{21}(1-x)-m_{12}x+c\left(\int y\wt\nu(t,y)dy-x\right)-s_1 x(1-x), \]
and $m_{12},m_{21}$ are the mutation rates $1\to 2$, $2\to 1$ respectively,  $c$ is the rate of
migration between colonies and the resampling rate $\gamma_1$ is inversely proportional to the
effective population size at a deme.  This model will be discussed in detail in subsection
\ref{sss.rand1} including Kimura's analysis which was based methods of ordinary differential
equations as well as the analysis using the dual representation which will be developed in section
3.  The duality method is not restricted to two-type systems (as is the case for the ode method)
and can be used to study general multitype systems.

\subsection{Multilevel multitype measure-valued processes}

The natural framework for multilevel multitype models with random effects at different
levels is the setting of multilevel measure-valued processes. Models of this type were first
developed for multilevel branching systems in Dawson and Hochberg \cite{DH-91} and Etheridge
\cite{ET-93}.  The long-time behaviour of multilevel measure-valued processes is investigated in Wu
\cite{Wu-94}, Dawson and Wu \cite{DW-96}, and Gorostiza, Hochberg and Waklbinger \cite{GHW-95}.  In
particular two-level measure-valued processes have state spaces of the form
$\mathcal{M}(\mathcal{M}(E))$ for some Polish space $E$ where $\mathcal{M}(E)$ denotes the space of
Borel measures on $E$.  In this paper we work with an analogous class of two level probability
measure-valued processes formulated in terms of a well-posed martingale problem which generalizes
the Kimura model to systems with more than two types, more complex interactions and to the
diffusion limit of systems with finitely many demes.

\subsection{Outline of the paper}

The main objectives of this paper are to formulate a general measure-valued framework for
multilevel population systems with mutation and selection and to develop the method of duality for
multilevel measure-valued stochastic processes with applications to population systems with
multilevel selection.  In section 2 we introduce the class of models characterized as solutions to
a well-posed  martingale problem. In Section 3 we introduce the dual processes used to establish
that the martingale problems are well-posed and to compute joint moments. These are given by the
multilevel generalization of the class of function-valued and set-valued dual processes introduced
in Dawson and Greven \cite{DG-14}. In Section 4 we consider the long-time behavior of systems with
two types and with two levels of selection. In Section 5 we introduce some more complex models of
systems with $K\geq 2$ types and with multilevel selection as well as further possible extensions
of these models and methods.

\section{Multitype-multilevel mutation-selection models}

\subsection{A two level finite population model}

We begin with a two-level finite population model given by an exchangeably interacting system of
Moran particle systems with selection at both levels. We assume that  the higher level  fitness of
a subpopulation can depend on the distribution of level I types within the subpopulation.  In other
words at the group level the fitness $V_2(\mu)$ of a subpopulation described by its distribution
$\mu$  over the space of types,  is a result of a network of interactions. Then the resulting
distribution of the collection of  subpopulations $\{\mu_i\}_{i\in S}$ is formulated in the setting
of multilevel measure-valued processes which provide a natural setting for the study of
hierarchical systems of this type.

We begin with a simple colony containing $N_1$ individuals described by a Moran model. \ Each
individual
has type in $\mathbb{I}=\{1,\dots,K\}.$ We let%
\bea{}
&& n_{k}    :=\text{ number of individuals of type }k\in\{1,\dots,K\}\\
&&N_1    =\sum_{k=1}^{K}n_{k}%
\eea and we think of the normalized vector $\frac{1}{N_1}(n_{1},\dots,n_{K})$ as an element of
$\mathcal{P}(\mathbb{I})$, the space of probability measures on $\{1,\dots,K\}$,
\[ X=\frac{1}{N_1}\sum_{k=1}^Kn_k\delta_k,\]
where the single atom measure $\delta_k$ represents an individual of type $k$.
\bigskip

The dynamics of a simple colony is given by a continuous time
Markov chain, $\{X_{t}:t\geq0\}$ with state space $\mathcal{P}%
(\mathbb{I}).$

The dynamics includes:
\begin{itemize}
\item Mutation:   given by transition rates $\{m_{ij}\}_{i,j\in \mathbb{I}}$, that is, the rate at which an
 individual of type $i$ is replaced by an individual of type $j$
\item Sampling: at rate $\frac{\gamma_1}{2}$  an individual of type $i$ is replaced by an individual of type $j$ where $j$ is chosen
 from the empirical distribution $X$
 \item Selection with fitness function $V_1:\mathbb{I}\to [0,1]$ and intensity $s_1$.
\end{itemize}

The resulting transitions for the probability-measure-valued process are given by%
\bea{}
&&\mu   \rightarrow\mu-\frac{1}{N_1}\delta_{i}+\frac{1}{N_1}\delta_{j}\text{ at rate \ }m_{ij}\mu(i)\\
&&\mu   \rightarrow\mu-\frac{1}{N_1}\delta_{i}+\frac{1}{N_1}\delta_{j}\text{ at rate
}((N_1-1)\frac{\gamma_1}{2}+s_{1}{V_{1}(j)})\mu
(i)\mu(j)%
\eea for $i,j\in \mathbb{I}$. Note that these assumptions result in a rate of change
${dX(t,i)}/{dt}$ due to mutation and selection of order ${1}/{N_1}$ which turns out to be the same
as the order of the sampling fluctuations when $\gamma_1>0$. This corresponds to the case of weak
selection in population genetics and the {\em diffusion limit}  below will involve a time speed-up
by a factor of $N_1$. In population genetics the parameter $\gamma_1$ is viewed as {\em inverse
effective population size} (see Remark 5.4 in \cite{D10}) and is a measure of the population size
in relation to the selection intensity in the finite population model.

We now consider a collection of $N_2$ colonies (demes) \ each consisting of $N_1$ individuals with
internal dynamics within each colony  given as above. In addition (as in the Aoki model) there is
an interaction between colonies via migration. To model this, individuals within a colony die and
are replaced at rate $c>0$  by a new individual with type  given by that of a randomly chosen
individual in a randomly chosen colony.

 The final mechanism is death (extinction) and replacement of colonies following the same
mechanism as the sampling mechanism within colonies. That is, a colony dies and is replaced by a
copy of a randomly chosen colony. \ In addition  we can include deme-level selection using a level
II fitness function $V_{2}(\mu)$ and selection intensity $s_2$. For
example, we can take a linear fitness function of the form%
\[
V_{2}(\mu)=\int v_{2}(x)\mu(dx)
\]

We then consider the empirical measure%
\[
\Xi_{N_2,N_1}(t):=\frac{1}{N_2}\sum_{i=1}^{N_2}\delta_{\mu_{i}(t)}\in \mathcal{P}%
(\mathcal{P}(\mathbb{I}))
\]
where $\mu_{i}$ denotes the state of the ith colony, namely,
\[ \mu_i=\frac{1}{K}\sum_{j=1}^{K} \frac{n_{i,j}}{N_1}\delta_j\]
and $n_{i,j}$ denotes the number of type $j$ individuals in the ith
colony.

The resulting transitions due to the level one dynamics are of the form

\[
\nu\rightarrow \nu +\frac{1}{N_2}(\delta_{\mu-\frac{\delta_i}{N_1}+\frac{\delta_j}{N_1}}-\delta
_{\mu})
\]%
at rate $m_{ij}\mu(i)+\left((N_1-1)\frac{\gamma_1}{2}+s_{1}V_{1}(j)\right)\mu(i)\mu(j)\nu(d\mu)$
\smallskip

The resulting transitions due to level II sampling and selection for the measure-valued process are given by:

\[
\nu\rightarrow(\nu+\frac{1}{N_2}(-\delta_{\mu_1}+\delta_{\mu_2}))\text{ rate
}\left(s_2{V_{2}(\mu_2)}+\frac{\gamma_{2}}{2}(N_2-1)\right)\nu(d\mu_1)\nu(d\mu_2)
\]

We can then consider the limiting behaviour as $N_1$ or $N_2$ go to $\infty$, or we can allow them
to go to infinity simultaneously with for example $N_2=\eta N_1$. In the following subsections we
will first consider the limit as $N_1\to\infty$ for a finite system of $N_2$ demes leading to a
system of interacting Fisher-Wright diffusions. We then consider the exchangeable system of
Fisher-Wright diffusions with additional death and random replacement of demes as described above
and then obtain a two level measure-valued diffusion process (called the two level Fleming-Viot
model) by letting $N_2\to\infty$.

\subsection{The martingale problem formulation}

The framework in which we specify the different stochastic models  is the class of
probability-valued Markov  processes $X(t)\in \mathcal{P}(E_1)$ where $E_1$ is a Polish space. Let
$D_{E_1}([0,\infty))$, ($C_{E_1}([0,\infty))$) denote the class of c\`{a}dl\`{a}g (resp.
continuous) functions from $[0,\infty)$ to $E_1$. We denote by $\{\mathcal{F}_t\}_{t\geq 0}$ the
natural filtration of $\sigma$-algebras on these spaces.
 The probability law   $P\in
\mathcal{P}(D_{E_1}([0,\infty)))$ is said to be a solution of the {\em martingale problem with
generator $(G,D(G))$}, where $G$ is a linear operator on $D(G)\subset C(E_1)$ and $D(G)$ is
measure-determining on $E_1$, if
 \bean{}\label{duality1} M_F(t):=&& F(X(t))-\int_0^t GF(X(s))ds \\&&\text{ is an }\mathcal{F}_t\text{-adapted }P\text{
martingale} \text{  for all }F\in D(G).\eean  The martingale problem method is used to characterize
stochastic processes of interest in many
applications. The method (which we will also use below) consists of four steps:\\
(1) to construct a sequence of approximating processes with probability laws $P_n\in$
$\mathcal{P}(D_{E_1}([0,\infty)))$ that satisfy
  some simple martingale problems,\\
(2) to show that the laws of the processes are tight, that is, relatively compact in
$\mathcal{P}(D_{E_1}([0,\infty)))$\\
(3) to show limit points of the $P_n$ satisfy the martingale problem defined by $(G,D(G))$, and
\\(4) to prove that there is a unique solution to this martingale problem thus characterizing the
limiting probability law $P$ of the process of interest. We will use this method to define the
Fleming-Viot process that models selection at two levels.
\bigskip

\noi A key tool used to establish the uniqueness of solutions is {\em duality}. This is achieved by
constructing a {\em dual process } $\mathcal{G}_t$ with state space $E_2$ and function
$F:\mathcal{P}(E_1)\times E_2 \to \mathbb{R}$ such that the functions $\{F(\cdot, g),g\in E_2\}$
are in $D(G)$ and are measure-determining on $E_1$, and  the duality relation: \bean{}
E_{X(0)}(F(X(t),\mathcal{G}_0))=E_{\mathcal{G}_0}(F(X(0),\mathcal{G}_t))\eean is satisfied for all
$\mathcal{G}_0\in E_2$ for all $X(0)\in \mathcal{P}(E_1)$ (where the right side denotes the
expectation with respect to the law of the $\mathcal{G}_t$ process). The class of set-valued dual
processes we use in the study of multilevel mutation selection systems is developed in detail in
Section 3. In addition to using the dual to establish that the martingale problem has a unique
solution it will be used below to compute moments, to obtain fixation probabilities and to prove
ergodicity.
\medskip

\noi
For general background on the martingale problem formulation see \cite{D93}, \cite{D10} and
\cite{EK2}.

\subsection{Diffusion process limits}

In this subsection  we will identify  the {\em  limit} of the process $\{\Xi_{N_2,N_1}(t)\}_{t\geq
0}$ as $N_1\to\infty$ for fixed $N_2<\infty$.


\subsubsection{The limiting single deme diffusion process}
\label{system1}  We first consider the special case in which $N_2 =1$ and let $N_1\to\infty$.
\begin{proposition}The limit as $N_1\to\infty$ of the single deme ($N_2=1$)  normalized empirical measure with
diffusion scaling leads (with time speed-up  $t\to N_1t$)) to  a $K$-type Fleming-Viot process
(equivalently, a finite type Wright-Fisher diffusion) which is characterized as the unique solution
of the martingale problem given with generator

\bea{}\label{G.0} G^{0}f(\mathbf{x})&&=\sum_{i=1}^{K}\left( \sum_{j=1}^K
(m_{ji}x_j-m_{ij}x_i)\right)\frac{\partial f(\mathbf{x})}{\partial x_i}\quad{\text{mutation}}\\&&
+s_1\,\sum_{i=1}^{K} x_i\left( V_1(i)- \suml^{K}_{k=1} V_1(k)x_k \right)\frac{\partial
f(\mathbf{x})}{\partial x_i}\quad{
\text{selection}}\nonumber\\&&+\frac{\gamma_1}{2}\sum_{i,j=1}^{K}
x_i(\delta_{ij}x_j-x_j)\frac{\partial^2f(x)}{\partial x_i\partial x_j}\quad{ \text{genetic drift}}
\nonumber\eea defined on the class $D(G^0)$ given by the class of functions $f$ with continuous
second derivatives on the simplex $\Delta_{K-1}=\{(x_1,\dots,x_K),\, x_i\geq 0,\; \sum x_i=1\}$.

\end{proposition}
\begin{proof}  See for example \cite{D93}, Theorem 2.7.1  for the neutral case and
\cite{D93} Theorem 10.2.1 for the proof of uniqueness for the case with selection. (Also see
\cite{EK2}, Chapter 10, Theorem 1.1 for the derivation of the diffusion limit starting with a
discrete generation model.)
\end{proof}

\subsection {Exchangeable system of  Wright-Fisher diffusions}

\noindent We now consider a system of demes labeled by $S=\{1,2,\dots,N_2\}$ where the population
at each deme undergoes mutation and selection as in the single deme process but in addition
individuals can migrate between demes at rate $c$ and the population in a deme can become extinct
at rate $s_2$ and be replaced with population $\mu$ sampled from
 the empirical distribution of deme compositions.  With selection, the replacement deme type is chosen
 with weights  proportional to the  level II fitness   \[0\leq V_2(\mu)\leq 1,\quad \mu\in\mathcal{P}(\mathbb{I}).\]

\bigskip

\subsubsection{Deme level fitness functions}

In order to incorporate fitness at the deme level we must introduce an appropriate class of fitness
functions.  It is natural to assume that the fitness of a deme (subpopulation) is a function of the
distribution of level I types within the deme given by $V_2(\mu)$ when the distribution of types
within the deme is $\mu\in\mathcal{P}(\mathbb{I})$.  We also assume that $V_2$ is a bounded and
continuous function of $\mu$ (in the topology of weak convergence).  Without loss of generality (by
the addition of a constant if needed) we can assume that $V_2(\mu)\geq 0$.

\begin{example} Consider the special case  $\mathbb{I} =\{1,2\}$, and
\be{}  V_2(\mu))=f(\mu(1))\geq 0.\ee Then (see e.g. Lorentz (1963) \cite{L-63} (Chapt. 1, Theorem
4)), we can uniformly  approximate $V_2$ using Bernstein polynomials as follows  \be{}
V_{2}(\mu)=\lim_{n\to\infty}\sum_{k,\ell}a_{n,k}(\mu(1))^k(\mu(2))^{n-k}\ee where the coefficients
$a_{n,k}\geq 0$.

\end{example}

In general, given a compact Polish space $E$ we consider the space $\mathcal{P}(E)$ of probability measures on
$E$ with the topology of weak convergence.  We then consider the Bernstein operators
$B^K:C(\mathcal{P}(E))\to C(\mathcal{P}(E))$ (where $C(\mathcal{P}(E))$ is a normed space with the
supremum norm) defined by \be{} B^Kf(\mu)=\int\dots\int
f\left(\frac{1}{K}\sum_{i=1}^K\delta_{x_i}\right)\mu(dx_1)\dots \mu(dx_K) \ee
 Then by (Dawson and G\"artner \cite{DGa} Theorem 3.9 ) for any $f\in
C(\mathcal{P}(E))$ \be{} B^Kf\to f\quad\text{in } C(\mathcal{P}(E)).\ee

This means that we can approximate any bounded continuous fitness function $V_2\in
C_+(\mathcal{P}(\mathbb{I}))$ by \be{} B^KV_2(\mu)=\int\dots\int
h(x_1,\dots,x_K)\mu(dx_1)\dots\mu(dx_K) \ee where $h$ is a bounded non-negative function on
$(\mathbb{I})^K$. This can be rewritten in the form \be{}\label{E.BA} B^KV_2(\mu)=\sum_i s_{i,K}
\int_{(\mathbb{I})^K} \prod_{j=1}^K 1_{A_{K,i,j}}d\mu^{\otimes K} \ee where for each $i$  the
$A_{K,i,j}$ are subsets of $ \mathbb{I}$. We denote the class of fitness functions of the form
(\ref{E.BA}) by $\mathcal{V}_K$ and note that we can approximate  any bounded continuous fitness
function by a functions in $\mathcal{V}:=\cup_K \mathcal{V}_K$.




\begin{example} Types $\mathbb{I}=\{1,2$\}. If $\mu(1) =p_1,\; \mu(2)=p_2$, then
\be{} V_2(\mu)=p_1p_2,\quad (1- V_2(\mu))=p_1^2+p_2^2+p_1p_2\ee \be{}  V_2(\mu)=\mu^\otimes (C), 1-
V_2(\mu)=\mu^\otimes(C^c)\ee where $C=\{1\}\otimes\{2\}$.


\end{example}

\begin{example}
Consider the 3 type case $\mathbb{I}=\mathbb\{1,2,3\}$ with fitness functions  as follows:
\[ V_1(1)=s_1,\quad  V_2(\mu)=s_2\mu(2)\mu(3).\]

\end{example}

\begin{example}\label{E4}   Model with  3 types  $\mathbb{I}=\{1,2,3\}$ and mutualistic (state-dependent) fitness.

 \begin{itemize}
 \item  $V_1(1,\mu)= s_1\mu(2)$,  $V_1(2,\mu)=s_1\mu(1)$, $V_1(3)=1/2$
 \item Level II fitness is $V_2(\mu)=s_2[\frac{1}{2}\mu(3)+ 2s_1\mu(1)\mu(2)]$.
 \end{itemize}
 This can be analysed using the set-valued dual as indicated in Remark \ref{R5}.
\end{example}

\begin{example}  $ V_2(\mu)$ is positive iff the population contains a certain set of properties
(from a finite set).

\be{} V_2(\mu)=\sum e_i\mu^\otimes(A_i)\ee \be{} 1- V_2(\mu)=\sum e_i\mu^\otimes(A^c_i)\ee where
$e_i\geq 0,\;\sum e_i=1$, $A_i\subset (\mathbb{I})^\N$.
\end{example}

\subsubsection{The limiting generator as $N_1\to\infty$ and $N_2<\infty$}

 \noindent The generator for the resulting model of $N_2$ interacting demes:
for $F\in C^2(\mathcal{P}(\mathbb{I})^{N2})$, with
$\mathbf{X}:=(\mathbf{x}_1,\dots,\mathbf{x}_{N_2})\in (\mathcal{P}(\mathbb{I}))^{N2}$

\bea{}\label{G.int} &&G^{N_2,\rm{int}} F(\mathbf{X}) \\&& = \; {{\eta}}\suml_{\xi=1}^{N_2}
G^{0}_\xi F(\mathbf{X}) \quad\quad\quad\text{ mutation-selection dynamics at each site}\nonumber
\\&&\; +
 c\cdot \suml_{\xi=1}^{N_2}\left[\sum_{j=1}^{K}\left(\sum_{\xi'=1}^{N_2}\frac{1}{N_2}\,x_j(\xi')
-x_j(\xi)\right)\frac{\partial F(\mathbf{X})}{\partial x_j(\xi)}\right] \quad{\text{migration}}\nonumber\\
&& +s_2\,\sum_{\xi=1}^{N_2}\left(\frac{1}{N_2}\sum_{\xi'=1}^{N_2} V_2(\mathbf{x}(\xi))
F(\Phi_{\xi\xi'}\mathbf{X})-F(\mathbf{X})]\right)\; { \text{deme replacement}}\nonumber\\&&
+\frac{1}{2}\gamma_2\sum_{\xi=1}^{N_2}\sum_{\xi'=1}^{N_2}[F(\Phi_{\xi\xi'}\mathbf{X})-F(\mathbf{X})]\quad
{ \text{deme resampling}}\nonumber\eea where $\Phi_{\xi\xi'}\mathbf{X} =(\mathbf{x}_1,\dots,\mathbf
{x}_\xi,\dots,\mathbf{x}_\xi, \mathbf{x}_{N_2}) $ (corresponding to the replacement of
$\mathbf{x}_\xi'$ by $\mathbf{x}_{{\xi}}$) and $\eta$ is a parameter that depends on the relation
between the natural time scales at the two levels.

\medskip

The  martingale problem with generator $G^{N_2,\rm{int}}$ has a unique solution that defines a
c\`adl\`ag strong Markov process $\{{\mathbf X}^{N_2}_t\}_{t\geq 0}$ with state space
$(\mathcal{P}(\mathbb{I}))^{N_2}$. The proof follows as in the proof of Proposition 2.1 but where
the dual process needed to show that the martingale problem is well posed in given in Subsection
3.1.

\subsection{Empirical measure-valued processes and the  Fleming-Viot limit}

We will next consider the limit as $N_2\to\infty$ in the general case in which we can have $s_2>0$
and/or $\gamma_2
>0$. We assume that the initial state satisfies $(\mu_1(0),\dots,\mu_{N_2}(0))\quad$ {is exchangeable}.

\beL{}\label{exchange} Consider the Markov process
$\mathbf{X}(t)=(\mathbf{x}_1(t),\dots,\mathbf{x}_{N_2}(t))\in (\mathcal{P}(\mathbb{I}))^{N_2}$ with
generator $G^{N_2,\rm{int}}$.  Assume that the probability distribution of $\mathbf{X}(0)$ is
exchangeable (i.e the distribution is invariant under permutations of $\{1,2,\dots,N_2\}$). Then
$(\mathbf{x}_1(t),\dots,\mathbf{x}_{N_2}(t))$ is an exchangeable system of
$\mathcal{P}(\mathbb{I})$-valued diffusions.
\end{lemma}
\begin{proof}  This follows since the migration and level II selection terms in the generator are
invariant under permutation - see \cite{vail} for the general case of exchangeable diffusions.

\end{proof}

{\bigskip}

 \noindent {\em The level II empirical process} is defined by
\medskip

\bea{}\label{nueq} \Xi^{N_2}_t :=\frac{1}{N_2}\sum_{j=1}^{N_2}\delta_{\mu_j}\in
\mathcal{P}(\mathcal{P}(\I)).\eea

\bigskip

\noindent Then by Lemma \ref{exchange}, $\Xi^{N_2}(t)$  is a
$\mathcal{P}(\mathcal{P}(\mathcal{\mathbb{I}}))$-valued Markov process with generator inherited
from the interacting system. To describe this we consider the algebra of functions, $D(G^{N_2})$,
on $\mathcal{P}(\mathcal{P}(\I))$ containing functions of the form

\be{}\label{alg} H(\nu)=\prod_{k=1}^K\left[\int h_k(\mu_k)\nu(d\mu_k)\right] \ee where

\be{}\label{alg2} h_k(\mu)   =\sum_j h_{k,j} \mu^{\otimes}(\prod_i 1_{A_{k,ij}})\qquad\text{that
is, a polynomial on }\mathcal{P}(\mathbb{I}).\ee

\noi We then define the generator in terms of the generator of the interacting system as follows:

\[ G^{N_2}H(\nu):=G^{N_2,\rm{int}} {F}(({\mu}_1,\dots,{\mu}_{N_2}))
\]
where $H\in D(G^{N_2})$ and \be{} \nu=\frac{1}{N_2}\sum_{j=1}^{N_2}\delta_{\mu_j},\ee and
$G^{N_2,\rm{int}}$ is given by (\ref{G.int}).

\begin{theorem}\label{T.1} (\cite{DG-93}, \cite{DG-14}).\\

\noindent Assume that $\gamma_2\geq 0$ and $V_2\in \mathcal{V}$. Then
 \bean{}  {\{\Xi^{N_2}_t\}_{t\in [0,T]}\Rightarrow (\Xi_t)_{t \in [0,T]} \;\text{  as  }N_2 \to \infty}\eean   where  $\Xi_t(dx) \in
C_{\mathcal{P}(\mathcal{P}(\mathbb{I}))}([0,T])$  is  the two level Fleming-Viot process with level
two selection given by the unique solution, $\{P_\nu:\nu\in\mathcal{P}(\mathcal{P}(\mathbb{I}))\}$,
to the well-posed $(G_2,{D}_2)$ martingale problem where {the domain ${D}_2\subset
C(\mathcal{P}(\mathcal{P}(\mathbb{I})))$} consists of the algebra of functions containing functions
of the form (\ref{alg}) and the generator acting on ${D}_2$ is given by

 \bea{}\label{G.2}
\\ G_2 H(\nu)&&=\int_{\mathcal{P}(\mathbb{I})} \eta\, G^0\frac{\delta
H(\nu)}{\delta\nu(\mu)}\nu(d\mu)\nonumber\\&&+  c
\int_{\mathcal{P}(\mathbb{I})}\int_{\mathbb{I}}\left( \frac{\delta}{\delta\mu_1(x)}\frac{\delta
H(\nu)}{\delta\nu(\mu_1)}\left[\int\nu(d\mu_2)\mu_2(dx)-\mu_1(dx)\right]\right)\nu(d\mu_1)\nonumber
\\&& +\frac{\gamma_2}{2}\int_{\mathcal{P}(\mathbb{I})}\int_{\mathcal{P}(\mathbb{I})}\frac{\delta^2
H(\nu)}{\delta(\nu(\mu_1))\delta(\nu(\mu_2))}\Big(\nu(d\mu_1)\delta_{\mu_1}(d\mu_2)-\nu(d\mu_1)\nu(d\mu_2)\Big)\nonumber\\&&
+s_2 \left[ \int_{\mathcal{P}(\mathbb{I})} \frac{\delta H(\nu)}{\delta \nu(\mu_1)}\left[V_{2}(\mu_1) -
\int_{\mathcal{P}(\mathbb{I})} V_2(\mu_2)\nu(d\mu_2)\right]\nu(d\mu_1)\right] ,\nonumber \eea where $G^0$
is given by (\ref{G.0}).

\end{theorem}
\begin{proof} We follow the standard argument which involves three steps: proof of the tightness of
the laws of the processes, proof of convergence of the generators on a sufficiently large class of
functions and finally proof that the martingale problem associated with the limiting generator has
a unique solution.  The first two steps follow in the usual way (e.g. proof of \cite{D93}, Theorem
5.3.1). It then remains to prove the uniqueness - this will be proved in the next section after
introducing the appropriate class of dual processes.

\end{proof}
\bigskip
\begin{remark}
An alternative class of functions, $\mathcal{D}_2$, is the linear span of functions of the form
\be{} H(\nu) = \prod_{k=1}^K\left(\int_{\mathcal{P}(\mathbb{I})} \left(\int_{\mathbb{I}^{n_k}}
h(x_{k,1},\dots,x_{k,n_k})\mu_k^{\otimes n_k}(dx_k)\right)\nu(d\mu_k)\right). \ee

 We also consider the convex set $\wt{\mathcal{D}}_2$ of
$[0,1]$-valued functions which contain functions of the above form with
 $h$ having values in $[0,1]$. Note that this class uniquely determines probability measures on $\mathcal{P}(\mathcal{P}(\mathbb{I}))$.

\end{remark}
\begin{remark}  The multilevel Fleming-Viot process is the analogue of the { multilevel superprocess}
- see e.g. Dawson-Hochberg (1991) \cite{DH-91},  Etheridge (1993) \cite{ET-93}, Wu (1994)
\cite{Wu-94}, Gorostiza-Hochberg-Wakolbinger (1995) \cite{GHW-95}, Dawson-Hochberg-Vinogradov
(1996) \cite{DHV-96}, Dawson and Wu (1996) \cite{DW-96}, Dawson-Gorostiza-Wakolbinger (2004)
\cite{DGW-04}.
\end{remark}

\bigskip





\begin{remark} In the special case  $s_2=0$ and $\gamma_2=0$ we obtain the mean-field limit.
We consider a tagged colony - by exchangeability this can be colony 1, $\mu_{1}^{N_2}(t)$. Then as
$N_2\rightarrow\infty$ in the limit we obtain  the measure-valued McKean-Vlasov dynamics (cf.
\cite{DG-14}) given by the solution to the  martingale problem with {\em nonlinear generator}

\be{} G_1^\nu F(\mu)=G^0F(\mu)+ c \int_{\mathbb{I}} \frac{\delta
F(\mu)}{\delta\mu(x)}[\int_{\mathcal{P} (\mathbb{I})}\nu(d\mu)(\mu(dx))-\mu(dx)]\ee and the law of
the process, $\Xi_t= \mathcal{L}(\mu_t)\in\mathcal{P}(\mathcal{P}(\mathbb{I}))$ is the weak
solution of a nonlinear second order partial differential equation.

If we assume $\gamma_2=0$ but $s_2>0$, then $\Xi_t$ is still deterministic and is the solution of a
nonlinear second order partial differential equation which is a generalization of Kimura's equation
(\ref{Kim}). Depending on the functions $V_1,V_2$ and with recombination these nonlinear equations can
 exhibit a range of behaviors including multiple equilibria and possible periodic or chaotic
 behaviour (see Akin \cite{A-83a}). In the general case with $\gamma_2 >0$ we obtain a two level Fleming-Viot process.

\end{remark}

\section{Duality for interacting and two-level Fleming-Viot systems}\label{duality}

In this section we introduce a basic tool, namely the generalization of the class of set-valued
dual introduced in \cite{DG-14} to the class of two level probability-measure-valued processes
$\Xi(t)\in \mathcal{P}(\mathcal{P}(\mathbb{I}))$ which were obtained in the previous section. These
processes satisfy the martingale problem with generator $G_2$.

 \bea{}\label{MP2} M_H(t):=&& H(\Xi_t)-\int_0^t G_2H(\Xi_s)ds \\&&\text{ is a } P-\text{
martingale} \text{  for all }H\in D_2(G_2).\nonumber\eea

\noindent The dual process developed here will be used to prove that there is a unique law \\ $P\in
\mathcal{P}(C_{\mathcal{P}(\mathcal{P}(\mathbb{I}))}([0,\infty)))$ which satisfies the martingale
problem (\ref{MP2}).  As explained above  the idea is to find a dual process $\mathcal{G}^2_t$ and
to establish the duality relation

\be{}\label{duality21}
E_{\Xi(0)}(F(\Xi(t),\mathcal{G}^2_0))=E_{\mathcal{G}^2_0}(F(\Xi(0)),\mathcal{G}^2_t)).\ee

\medskip

We begin by obtaining the dual for the system of interacting $\mathcal{P}(\mathbb{I})$-valued
processes with generator $G^{N_2,int}$ given by (\ref{G.int}).  For detailed background on the
duality method  to be used refer to \cite{DG-14} Chapter 5.

\medskip
In

\subsection{A function-valued dual}

We now introduce a function-valued dual for the process with  generator $G^{N_2,int}$.
 The state space for the function-valued dual is the set of functions, $\mathbb{H}$, of the form\\ $\sum_k
\prod_{i=1}^{N_2}\prod_{j=1}^{n_i} h_{k,i,j}(x_{ij})$. By inspection of the action of the generator
$G^{N_2,int}$ on functions in $\mathbb{H}$, we can read off the corresponding function-valued
transitions corresponding to mutation, selection and migration  as follows.

 \begin{itemize}
\item
Level I Selection with $V_1(x)=1_B(x)$.

\noindent Transitions at rate $s_1$

\be{} h(x_1,\dots,x_n)\to 1_B(x_i)h(x_1,\dots,x_n)+1_{B^c}(x{n+1})h(x_1,\dots,x_n)\ee

\item Mutation

\be{} h(\dots,x_i,\dots)\to \int h(\dots,y,\dots)M(x_i,dy)\ee

\item  Level I Coalescence: At rate $\frac{\gamma n(n-1)}{2}$,

 \be{} h(x_1,\dots,x_n)\to h(x_i,\dots,x_i,\dots, x_{j-1},x_i,x_{j+1},\dots,x_n)\ee

\item  Migration:  For each $i,j\in S$, at rate $\frac{c}{N_2}$ ,  \be{}\label{mig} h_1(x_{i1},x_{i2})h_2(x_{j1},x_{j2})\to
h_1(x_{i1},x_{j3})h_2(x_{j1},x_{j2}) \ee Here the first index indicates the deme and the second the
rank at the given deme.

\item  Level II selection: By (\ref{E.BA}) and taking convex combinations  it suffices to consider a level II fitness function of the form:
\[  V_2(\mu)= \mu(B)\]

\bea{}&&   h(\cdot)\longrightarrow   V_2(\cdot) h(\cdot) +(1- V_2(\cdot))\otimes  h(\cdot) \eea

\noi Then the level II selection: for each $i,j\in S$ in transitions
 at rate $\frac{s_2}{N_2}$

\[ h(x_{i1},x_{i2})\to 1_B(x_{i1})h(x_{i2},x_{i3})+(1-1_B(x_{j1}))h(x_{i1},x_{i2})\]

\item Level II coalescence: for each pair $i,j$  at rate $\frac{\gamma_2}{2}$ \be{}
h_1(x_{i1},x_{i2})h_2(x_{j1},x_{j2})\to h_1(x_{i1},x_{i2})h_2(x_{i3},x_{i4}).\ee

\end{itemize}

\subsection{A set-valued dual for exchangeably interacting systems of  Fleming-Viot
processes}\label{sss.sv}
\medskip

We now introduce the set-valued dual which will be used to study the interacting system of
Fleming-Viot processes and then the limiting two-level Fleming-Viot process. This is based on  the
set-valued dual introduced in \cite{DG-14} (subsections 9.4, 9.5) for the system of exchangeably
interacting Fleming-Viot processes but extended in order to include level II selection and
resampling.
\medskip

\noi We begin with the population at a set of demes labeled by $S$ with \be{} S=\{1,\dots,N_2\}\ee
with migration between demes as defined in subsubsection \ref{system1} with the assumption of
exchangeability.  Recall that the state space for the finite system of interacting Fleming-Viot
processes is $(\mathcal{P}(\mathbb{I}))^{S}$.The set-valued dual is a refinement of the
function-valued dual sketched above. Noting that it suffices to work with linear combinations of
indicator function the Level I function-valued and set-valued version of the above dual were
introduced and studied in depth in Dawson and Greven \cite{DG-14}.

We now introduce the state space and notation needed to define  the set-valued dual
$\mathcal{G}_t$.

 \noi Recall that $\mathbb{I}=\{1,\dots,K\}$. We indicate the indicator function of a subset
$A\subset \mathbb{I}$ by $1_A=(e_1,\dots,e_K)$ with $e_i=1$ if $i\in A$ and $e_i=0$ is $i\in A^c$,
that is, the complement of $A$. For example, the indicator function of $\{1,2\}\subset\{1,2,3\}$ is
indicated by $(110)$. We sometimes identify finite subsets with their indicator functions.

Let
\begin{eqnarray}&&\mathcal{I}:= \text{ algebra of subsets of }\mathbb{I}^{\N}\nonumber
\\&&\qquad\text{  of the form } A\otimes_1 \mathbb{I}^{\N},\; A\text{ is a subset of  }\mathbb{I}^m,m\in\N,\nonumber\end{eqnarray}
the coordinates in a product in $\mathbb{I}^m$ are called {\em ranks}. Given $A,B\subset
\mathbb{I}$ we denote the product of these sets in $\mathbb{I}\times\mathbb{I}$ as $A\otimes_1 B$.
Given $A,B\subset \mathcal{I}$ we denote the product of these sets in
$\mathcal{I}\times\mathcal{I}$ as $A\otimes_2 B$.
\bigskip

\noi The {\em state space:} $\sf{I}^{N_2}$ for the set-valued dual associated to the interacting
systems of Fleming-Viot processes with $S= \{1,\dots,N_2\}$ is the algebra of sets containing sets
of the form

 \bea{}\label{fnc1} && \bigotimes_{2,\,i\in S} \left(\otimes_{1\,,j=1}^{n_i} A_{i,j}\right),\quad
A_{i,j}\in\mathbb{I},\; n_i\in\N,\\&&\in (\mathcal{I})^{\otimes_2 S}.\nonumber\eea

In order to describe the dual dynamics we first describe the transitions that occur for a set
written as a disjoint union  of  sets of the form (\ref{fnc1})
where in $A_{i,j}\subset \mathbb{I}$ the first subscript denotes the deme and the second subscript
denotes the rank at the deme,  and $V_1(x) = 1_{B}(x)$ with rate $s_1>0$.
\bigskip

 \noi The  transitions of the set-valued process $\mathcal{G}^{N_2,int}_t$ are  obtained by restricting the function-valued
 transitions to indicator functions of sets in $\sf{I}^{N_2}$.  These are then given by:
 \smallskip

\noindent Level I selection at rank $j^*$ at deme $i^*\in S$ at rate $s$:  \be{}\label{sel1}
A_{i^*.j^*} \to B_{i^*.j^*}\cap A_{i^*.j^*}\cup B_{i^*.j^*}^c\otimes_1 A_{i^*.j^*+1}\ee and all
other ranks larger than $j^*$ are also shifted to the right at deme $i^*$.

\bigskip

\noindent Mutation at rank $j$ at deme $i$: (refer to \cite{DG-14}, Definition 5.12 and
Subsubsection 9.5.3) \be{}A_{ij}\subset\mathbb{I}\to A_{ij}\cup\{\ell\}\text{  with } \ell\in
\mathbb{I}\text{ at rate }\sum_{k\in A}m_{\ell,k},\ee or \be{}A_{ij}\to A_{ij}\backslash
\{\ell\}\text{  at rate }\sum_{k\in A_{ij}^c}m_{\ell,k}.\ee
\bigskip

\noindent Coalescence at rate $\gamma_1/2$ of ranks $j_1$ and $j_2>j_1$ at deme $i\in S$:
$A_{i,j_1}\otimes_1 A_{i,j_2}\to \wt A_{i,j_1}= A_{i,j_1}\cap A_{i,j_2}$ and $\wt A_{i,j}=
A_{i,j+1}$ for $j\geq j_2$.

\bigskip

\noindent Migration at rate $\frac{c}{N_2}$ of rank $j$ from deme $i_2\in S$ to $i_1\in S$. Let
$A_i=\otimes_{1,i=1}^{n_i}A_{ij}$.

 \be{}A_{i_1}\otimes_2
A_{i_2}\to \wt A_{i_1}\otimes_2 \wt A_{i_2}\ee with \be{} \wt A_{i_1,n_1+1}= A_{i_2,j}\ee \be{}\wt
A_{i_2,\ell}=A_{i_2,\ell+1}\quad\text{for }\ell\geq j\ee
 \be{}\wt
A_{i_2,\ell}=A_{i_2,\ell}\quad\text{for }\ell < j\ee

\begin{remark}

 Note that in the limit $N_2\to\infty$ the measure $\nu$ is nonatomic and migration or level II
selection transitions always lead to a new (that is unoccupied)  deme.
\end{remark}
\bigskip

\medskip

\noi {\em Coupling:} Note that every set in $\sf{I}^{N_2}$  can be written as the union of a finite
number of disjoint sets of the form $\otimes_{2,i=1}^N\otimes_{1,j=1}^{n_i}A_{i,j}$ with
$A_{i,j}\subset \mathbb{I}$ and $N\in\N$. Finally the above transitions are simultaneously carried
out in this disjoint union of products and are coupled as follows: all selection, mutation,
coalescence and migration operations are simultaneously applied to each rank at each deme of each
product in the disjoint union.  Each such transition preserves the decomposition of the disjoint
union into a new disjoint union - this is obviously satisfied for mutation, coalescence and
migration and true for selection in view of the specific form (\ref{sel1}).
\bigskip

\subsubsection{The dual representation}


We now state the duality relation between the system of interacting Fleming-Viot processes
$\mathbf{X}$ under the assumption

 \be{} \mathbf{X}(0)=\mu^{\otimes_{2}N_2}.\; \text{  with }
\mu\in\mathcal{P}(\mathbb{I})\ee which implies that we have a system of exchangeably interacting
Fleming-Viot processes.
\medskip

 Define the function $F:\mathcal{P}(\mathbb{I})\otimes \sf{I}\to [0,1]$ by
 \bea{} F(\mathbf{X},\mathcal{G})=  \mathbf{X}^*(\mathcal{G})\eea
 where if $\mathbf{X}(0)=\otimes_{2,j=1}^{N_2} \mu_j$, then { $\mathbf{X}^*(0)
 =\otimes_{2,j=1}^{N_2}(\mu_j)^{\otimes_1\N}\in \mathcal{P}((\mathbb{I}^{\mathbb{N}})^{N_2})$}.  For
 example,
 if $G=\bigotimes_{i\in S} G_i$ with $G_i\in\mathcal{I}$, then
 \be{} \mathbf{X}^*(G)=\prod_{j\in S} \mu_j^{\otimes_1 \N}(G_j).\ee

\bigskip

\begin{theorem}\label{T.2} Let $\mathbf{X}^{N_2}$ denote a solution to the martingale problem with
generator $G^{N_2,int}$. Then  \noi (a) \noindent{Dual Representation}{
 \bea{}\label{dr1} E_{\mathbf{X}(0)}(F(\mathbf{X}^{N_2}_t,\mathcal{G}^{N_2,int}_0))=E_{\mathcal{G}^{N_2,int}_0}(F(\mathbf{X}^{N_2}_0,\mathcal{G}^{N_2,int}_t))\eea}

\noi (b) The representation (\ref{dr1}) uniquely determines the marginal distribution of the
process $\mathbf{X}^{N_2}(t)$ and therefore establishes the uniqueness of the solution to the
martingale problem.
\end{theorem}

\begin{proof}  The proof in the case $s_2=0$ is given in detail in \cite{DG-14} based on verifying that the generators
of the two processes acting on the function $F$ satisfy the relation \be{}
G^{N_2,int}F(\mu,A)=G^{dual}F(\mu,A)\quad \text{  for all  } \mu\in \mathcal{P}(\mathbb{I}),\; A\in
\sf{I},\ee where $G^{dual}$ is the generator of the set-valued Markov jump process with transition
rates given above. The extension to the case $s_2>0$ follows in the same way and will be given in
more detail for the two-level process below.
\end{proof}

\bigskip
\begin{remark}\label{R5}  \textbf{State-dependent fitness}

We can also consider level I selection that is state dependent, that is in which the fitness  of a
type depends on the distribution of types (e.g. diploid). For example we could have the fitness of
type $1$ proportional to the  population proportion of type $B$, that is, $V(1,\mu)= s\mu(B),\;
B\subset \mathbb{I},\; s\geq 0$. In this case the dual has function-valued transitions at rate $s$:
\bean{}
&& f\to 1_B\otimes_1 [1_1f-1_1\otimes_1 f]+f\\
&& =1_B\otimes_1[1_1f +(1-1_1)\otimes_1 f]\\
&& \quad + (1-1_B)\otimes_1 f. \eean

A second example to a set-valued dual which can be used to analyse such systems, eg. mutualistic
types (see Example \ref{E4}).  Consider $\mathbb{I}=\{1,2,3\}$. \[ V_1(1)=v_1,\quad V_1(2,\mu)=
v_{M}\cdot \mu(3), \;V_1(3,\mu)=v_{M}\cdot \mu(2),\] with transitions \begin{eqnarray}\label{sdf}&&
f \to 1_1f+(1_{2}+1_3)\otimes_1 f\quad \text{ at rate }v_1,\\&&
 f\to 1_3\otimes_1[1_2f+(1_{1}+1_3)\otimes_1 f]+(1_1+1_2)\otimes_1 f \quad \text{ at rate }v_M,\nonumber\\&&
f\to 1_2\otimes_1[1_3f+(1_{1}+1_2)\otimes_1 f]+(1_1+1_3)\otimes_1 f \quad \text{ at rate
}v_M.\nonumber\end{eqnarray}

\end{remark}

\subsection{A set-valued dual for the two level Fleming-Viot process}

The objective of  this subsection is to extend the set-valued dual of subsection \ref{sss.sv} in
order to construct set-valued duals $\mathcal{G}^{2,N_2}_t$, $\mathcal{G}^2_t$  for the
$\mathcal{P}(\mathcal{P}(\mathbb{I}))$- valued processes $\{\Xi^{N_2}(t)\}_{t\geq 0}$ (assuming
exchangeable initial configuration) and $\{\Xi(t)\}_{t\geq 0}$.  We assume that the level II
selection  rate is $s_2$ and with fitness function $V_2$ and the level II resampling  rate is
$\frac{\gamma_2}{2}$.  To simplify the notation we take $\eta=1$ in (\ref{G.int})  in the
subsequent discussion.
\bigskip

\noi The {\em state space:} $\sf{I}^{2*}$ for the set-valued dual $\mathcal{G}^2_t$ is  is the
algebra of subsets of $(\mathbb{I}^{\N})^{\N}$ containing sets of the form

 \bea{}\label{fnc12} && \bigotimes_{2,\,i\leq m} \left(\otimes_{1\,,j=1}^{n_i} A_{i,j}\right),\quad
A_{i,j}\in\mathbb{I},\; n_i\in\N,\\&&\in (\mathcal{I})^{\otimes_2 m},\quad m\in \N,\nonumber\eea
where $i$ is the index of the deme and $j$ is the index of the rank within the deme.
 The transitions of the dual due
 to level I mutation, resampling, and selection at each deme and migration between demes are given as
above in subsection \ref{sss.sv}. In the $N_2\to\infty$ limit, migrants always move to a new
(unoccupied) deme, namely the the unoccupied deme of lowest index.

\medskip

\noi  We assume that $V_2$ belongs to the class of level II fitness functions of the form
 \be{}\label{v2rep} V_2(\mu)= \sum_j s_{2,j} V_{2,j}(\mu),\;\;
 V_{2,j}(\mu)= \mu^{\otimes}( 1_{B_j})\ee and $B_j=\prod_{i=1}^{n_j}B_{ji}$ with $B_{ji}\subset
\mathbb{I}$.
\bigskip

\noi  We
now introduce the additional transitions that occur due to Level II selection and coalescence. For
the former, using the linearity
 it suffices to describe the contribution to the dynamics of $ V_{2,j}$ with $
V_{2,j}(\mu)=\mu^\otimes(B)$, $B\in\mathcal{I}$.
\bigskip

\noindent  As above  we use the notation $\otimes_1$ and $\otimes_2$ to distinguish such products
on $\mathbb{I}$ and $\mathcal{I}$. Similarly for measures in
$\nu_i\in\mathcal{P}(\mathcal{P}(\mathbb{I}))$  we write  $\nu_1\otimes_2\nu_2$ for the product
measure.

\bigskip

\noindent \underline{Set-valued transitions - deme level selection}
\smallskip

 Using the linearity
 it suffices to consider the contribution to the dynamics of $ V_{2,j}$ with $
V_{2,j}(\mu)=\mu^\otimes(B_j)$, $B_j\in(\mathbb{I})^{n_j}$.  Given such a fitness function and sets
$A_i\in\mathcal{I}$ with the subscript indicating the deme, then for every $k,\ell\in S$ the action
of selection on deme $k$  results in the  transition

\be{} \bigotimes_{2, i\in S} A_i\to \left(B_k\otimes_1 A_k \otimes_2 \bigotimes_{2,\,i\ne
k}A_i\right) \bigcup \left(B^c_\ell\otimes_2  \bigotimes_{2,\, i\in S} A_i\right)\ee occurs with
rate $\frac{s_{2,j}}{N_2}$ if $S=\{1,\dots,N_2\}$.
\medskip

If $S=\mathbb{N}$, then this becomes

\be{}\label{22sel} \bigotimes_{2,\, i=1}^{n} A_i\to \left(B_k\otimes_1 A_k \otimes_2
\bigotimes_{2,\,i\ne k}A_i\right) \bigcup \left(B^c_\ell\otimes_2  \bigotimes_{2,\,i=1}^{n}
A_i\right)\ee where $A_i,B_i\in\mathcal{I}$, $n$ is the number of occupied demes (i.e. not
identically $\mathbb{I}^\N$)  and $\ell$ denotes the first unoccupied deme, occurs with rate
$s_{2,j}$.
\bigskip

\noi Note that exchangeability is preserved by the dynamics and we can again {\em couple} the
corresponding indices after a level II selection event, that is all further transitions are applied
simultaneously to the corresponding indices (deme and rank at the deme).  For this reason we can
rewrite (\ref{22sel}) as

\be{}\label{l2selx} \bigotimes_{2,\, i=1}^{n} A_i\to \left(B_k\otimes_1 A_k \otimes_2
\bigotimes_{2,\, i\ne k}A_i\right) \bigcup \left(B^c_k\otimes_2  \bigotimes_{i} \wt A_i\right)\ee
where $\wt A_i =A_i\text{ if }i<k$, $\wt A_i =A_{i+1}\text{ if }i\geq k$.  This means that the new
event is a disjoint union of events and this is preserved by all further transitions.
\bigskip

\noindent Example: Let $\mathbb{I}=\{1,2\}$ with fitness $V_2(\mu)=\mu(1)$. Then given
$1_{\mathcal{G}}= 1_1$, then the first transition (in terms of indicator functions)
\[  1_1\to 1_1\otimes_1 1_1+ 1_1\otimes_2 1_2\]
so that letting
\[ E_\nu(\mu(1))=\int \mu(1)\nu(d\mu),\] we have

\[ E_\nu(\mu(1))\to  E_\nu(\mu(1)) +E_\nu((\mu(1))^2)-(E_\nu(\mu(1)))^2= E_\nu(\mu(1))+\text{Var}_{\nu}(\mu(1))\]
so that if $\nu$ is not a single atom probability measure the level II selection is effective.
\bigskip

\noi \underline{Set-valued transitions - deme level coalescence}

The level II resampling results in the coalescence of two demes, for example demes $1,2$ as
follows:

\bean{}&& \left(\bigotimes_{1,j} B_j\right)_1\otimes_1 \left(\bigotimes_{1,i} A_i\right)_1 +
\left(\bigotimes_{1,j}
B_j\right)_1^c\otimes_2 \left(\bigotimes_{1,i} A_i\right)_2\\
&&\to \left(\bigotimes_{1,j} B_j\right)_1 \otimes_1 \left(\bigotimes_{1,i} A_i\right)_1+
\left(\bigotimes_{1,j} B_j\right)_1^c\otimes_1 \left(\bigotimes_{1,i} A_i\right)_1\\&&=
\left(\bigotimes_{1,i} A_i\right)_1\eean where the exterior subscripts denote the deme.

When there is no level II coalescence and $\nu_0$ is non-random, the  $\{t\to \nu(t)\}$ is
deterministic and it suffices to consider $k_0=1$.  The reason is that there is no interaction
between the supports of the associated set-valued processes starting with disjoint supports in this
case so that $Var(\int h(\mu)\nu_t(d\mu))=0$.

\bigskip

\subsubsection{The duality relation for the  two level Fleming-Viot process and its applications}

We now present the dual representation for the two level  Fleming-Viot systems.
\medskip

\noi Let  $\{\Xi(t)\}_{t\geq 0}$ denote a $\mathcal{P}(\mathcal{P}(\mathbb{I}))$-valued process
with probability law \[P^{\Xi_0}=\mathcal{L}(\Xi)\in
\mathcal{P}(C_{\mathcal{P}(\mathcal{P}(\mathbb{I})))}([0,\infty)))\] which satisfies the
$(G_2,{{D}}_2)$-martingale problem. Then the time-marginals $\Xi(t)$ are random probability measure
on $\mathcal{P}(\mathbb{I})$. Then by de Finetti's theorem (see \cite{D93}, Theorem 11.2.1) there
exist a sequence $\{\wt\mu_n\}$ of $\mathcal{P}(\mathbb{I})$-valued exchangeable random variables
such that $(\wt\mu_1,\dots,\wt\mu_n)$ has joint distribution \be{}
P^{(n)}(t,d\wt\mu_1,\dots,d\wt\mu_n)=\int_{\mathcal{C}_{\mathcal{P}(\mathcal{P}(\mathbb{I}))}[0,\infty)}\Xi(t,d\mu_1)\dots\Xi(t,d\mu_n)\,dP^{\Xi_0},\quad
n\in\N,\ee that correspond to the moment measures of $\Xi(t)$.

\medskip

\noi Let $\mathcal{G}_t$ with values in $\sf{I}^{2*}$ denote the set-valued process defined above.
\medskip

 \noi Define the function $ \mathcal{H}:(\mathcal{P}(\mathcal{P}(\mathbb{I})))\times (\mathcal{I})^{\otimes \N,*}\to [0,1]$ by
 \bea{} &&\mathcal{H}(\nu,\mathcal{G})= \sum_{k=1}^{N_1}\prod_{i=1}^{N_{2,k}}\int \left[\mu_i^{\otimes
n_{ki}}(A_{k,i})\right]\nu(d\mu_i)\;\;\text{  if  }
 \mathcal{G}=\bigcup_{k=1}^{N_1}\bigotimes_{2,i=1}^{N_{2,k}} A_{k,i},\text{  with  } A_{k,i}\subset (\mathbb{I})^{n_{ki}},\nonumber\\&&
 \text{with  } \bigotimes_{2,i=1}^{N_{2,k_1}} A_{k_1,i}\cap \bigotimes_{2,i=1}^{N_{2,k_2}} A_{k_2,i}=\varnothing\text{  if  }k_1\ne k_2,\nonumber \\&&
 \text{we can also write this as }\mathcal{H}(\nu,\mathcal{G})=\nu^*(\mathcal{G}).\nonumber\eea

 \beT{}\label{T.3}{Dual Representation}

(a)  For any solution $\{P_{\Xi_0}:\Xi_0\in\mathcal{P}(\mathcal{P}(\mathbb{I}))\}$ of the
$(G^{N_2},{{D}}_2)$ or $(G_2,{{D}}_2)$-martingale problem
 \bea{}\label{drep2} E_{\Xi_0}(\mathcal{H}(\Xi_t,\mathcal{G}^2_0))=E_{\mathcal{G}_0}(\mathcal{H}(\Xi_0,\mathcal{G}^2_t))\eea

 (b) The $(G^{N_2},{{D}}_2)$ and $(G_2,{{D}}_2)$-martingale problems are
 well-posed.
\end{theorem}

\begin{proof}  The proof for the cases $(G^{N_2},{{D}}_2)$ and $(G_2,D_2)$ follow the same lines.  We now give the details
the the latter case.\\    (a)  As above, we begin by identifying the terms in $G_2H(\nu)$  for
functions
 in ${D}_2$ of the form: \be{}H(\nu)=\prod_{k=1}^{k_0}\left(\int_{\mathcal{P}(\mathbb{I})}
\left[\int_{\mathbb{I}^{n_k}} h(x_{k,1},\dots,x_{k,n_{k}})d\mu_{k}^{\otimes
n_{k}}\right]\nu(d\mu_{k})\right),\ee in other words we work with functions of the form
\[ \prod_{k=1}^{k_0}h(x_{k,1},\dots,x_{k,n_k}).\]

\medskip
The transitions  for functions of this form are given as follows:

\begin{itemize}
\item Level I resampling. This results in the coalescence \be{}\int \left[\int\int
h(x_{11},x_{12})\mu(dx_{11})\mu(dx_{12})\right]\nu(d\mu)\to \int\left[\int
h(x_{11},x_{11})\mu(dx_{11})\right]\nu(d\mu),\quad \text{at rate }\gamma_1\ee
\item Migration. At rate $c$
\bea{}&& \int \left[\int h(x_{11},x_{12})\mu(dx_{11})\mu(dx_{12})\right]\nu(d\mu)\\&& \to
\int\int\left[ \int
h(x_{11},x_{22})\mu_1(dx_{11})\mu_2(dx_{22})\right]\nu(d\mu_1)\nu(d\mu_2).\nonumber\eea

\item Selection at level I with

\be{} V_1(x)= s_1 1_{B}(x).\ee This results in the transition with

\bea{}&&\int \int h(x_{11})\mu(dx_{11})\nu(d\mu)\\&&\to \int\left[\int
h(x_{11})1_B(x_{11})\mu(dx_{11})+\int\int
1_{B^c}(x_{11})h(x_{12})\mu(dx_{11})\mu(dx_{12})\right]\nu(d\mu)\nonumber \eea
 at rate $s_1$.
\end{itemize}

\medskip

\noi Under {\em level II coalescence} two occupied demes $i,j \in S,\,i\ne j$ are chosen at random
at rate $\gamma_2/2$ and we have \bea{} &&\\&&\int\left[ \int\int
h(x_{i1},x_{i2})h(x_{j1},x_{j2})\mu_i(dx_{i1})\mu_i(dx_{i2})\mu_j(dx_{j1})\mu_j(dx_{j2})\right]\nu(du_i)\nu(d\mu_j)\nonumber\\&&\to
\int\left[ \int \int\int\int
h(x_{i1},x_{i2})h(x_{i3},x_{i4})\mu(dx_{i1})\mu(dx_{12})\mu(dx_{i3})\mu(dx_{i4})\right]\nu(d\mu),\nonumber\eea
\[ h\otimes_2h\to h\otimes h.\]

In particular we have
\[ \int\left(\int h(x)\mu(dx)\right)\nu(d\mu)\cdot \int\left(\int h(x)\mu(dx)\right)\nu(d\mu)\to
\int \left(\int h(x)\mu(dx)\right)^2\nu(d\mu).\]

\medskip

\noi {\em Level II selection} with fitness function $ V_{2}$

Now consider the  case in which $h(.)$ and $ V_2(\cdot)$ are polynomials, that is,

\be{}\label{hdef}  h(\mu)   =\sum_j  h_j \mu^{\otimes}(\otimes_{1,i}{A_{ji}})\qquad\text{polynomial
on }\mathcal{P}(\mathbb{I}),\; h_j\leq 1\ee where \be{}\label{v2rep2} V_2(\mu)= \sum_j a_j
V_{2,j},\;\; V_{2,j}= \mu^{\otimes}(\otimes_{1,i}{B_{ji}})\ee

\be{}\label{hsela}  h(\mu)\to V_{2,j}(\mu_1) h(\mu_1)+(1- V_{2,j}(\mu_1)) h( \mu_2)\quad\text{at
rate } a_j.\ee

\be{} \label{hselb}\int h(\mu)\nu(d\mu)\to\int V_{2,j}(\mu_1)h(\mu_1)\nu(d\mu_1)+\int\int
(1-V_{2,j}(\mu_1))h( \mu_2)\nu(d\mu_1)\nu(d\mu_2)\ee

\bea{}&& h(x_{11},\dots,x_{1,n_1})\to\\&&  V_2(x_{1,1},\dots,x_{1,n_2})
h(x_{(1,n_2+1)},\dots,x_{1,(n_2+n_1)})\nonumber\\&&+(1- V_2(x_{1,1},\dots,x_{1,n_2}))
 h(x_{2,(n_2+1)},\dots,x_{2,(n_2+n_1)})\nonumber\eea

\bigskip

\noi More generally, if $H(\nu)=\int\prod_{i=1}^K  h(\mu_i)\nu(d\mu_i)$ Assume $ V_2\leq 1$, namely
the indicator function of a set. The selection acting on $H$ produces
\bea{}\label{hselc} &&  \int\prod_{i=1}^K  h(\mu_{i})\nu(d\mu_{i}) \longrightarrow\\
&& \sum_{j=1}^K \Big(\Big\{ \prod_{i\ne j}h(\mu_{i})\nu(d\mu_{i})\Big\}\nonumber
\\&&\cdot \left[\int  h(\mu_{j}) V_2(\mu_{j})\nu(d\mu_{j})+\int (1-  V_2(\mu_j))\nu(d\mu_j)\int  h(\mu_{K+1})\nu(d\mu_{K+1})\right]\Big)\nonumber\eea

\[
h(\mu_1)\to V_2(\mu_1) h(\mu_1)+ (1- V_2(\mu_1))\otimes  h(\mu_2).\]

\bigskip

\noi The corresponding set-valued transitions are obtained by restricting the class of functions
$h$ of the form $h(\mu)=\mu^\otimes(A)$ with $A\in\mathcal{I}$  as in subsubsection \ref{sss.sv}.

\medskip

\noi \underline{Coupling}. In view of the assumption of exchangeability, we can couple the $
V_2(\mu) h(\mu)$ and $(1-V_2(\mu))$ terms for the level II selection transitions  at the deme
level, that is, place these at the same deme index in the two resulting summands thus producing a
union of disjoint sets in $\sf{I}^{2*}$. Then as before, all operations are performed
simultaneously on all demes and ranks in the different summands.
\bigskip

\noi \underline{Set-valued transitions}  The set-valued transitions  can then be read off by
restricting the function-valued transitions  to the class of functions $h$ that are based on
indicator functions of sets as in (\ref{hdef}) and noting that due to the coupling the transitions
preserve the decomposition into the union of {\em disjoint}  subsets. These transitions define a
Markov jump process $\{\mathcal{G}^2_t\}_{t\geq 0}$ with countable state space $\sf{I}^{2*}$ and we
denote the resulting generator  by $G^{dual}$. The identity of the action of the corresponding
terms of $G_2$ acting on $\mathcal{H}(\nu,\mathcal{G})$ and the result of the transition of the
set-valued dual, that is,
 \be{}\label{G-G}
G_2 \mathcal{H}(\nu,\mathcal{G})=G^{dual}\mathcal{H}(\nu,\mathcal{G})\quad \text{  for all  }
\nu\in \mathcal{P}(\mathcal{P}(\mathbb{I})),\; \mathcal{G}\in (\mathcal{I})^{\otimes \N,*},\ee is
then immediate by inspection. For example, setting $V_2(\mu)=\mu^\otimes(\otimes_{1,i} B_i)$ and
applying the corresponding selection transition from $G_2$ to  $H(\nu)=\int \prod_{i=1}^K
h_i(\mu_i)\nu(d\mu_i)$ with $h_i(\mu_i)= \mu_i(\otimes_{1,j} A_{ij})$, (\ref{hselc}) yields the
set-valued transitions (\ref{22sel}) with $A_i=\otimes_{1,j} A_{ij}$ which correspond the
$G^{dual}$.
\medskip

\noi   The duality relation (\ref{drep2}) then follows from (\ref{G-G}) (see for example
Proposition 7.10 in \cite{D10} or Chapter 4 of \cite{EK2})). \medskip

 \noi The uniqueness of the solution to the martingale problem is then obtained. In particular,
moment measures of the time marginals  of any solution to the martingale problem are determined by
the dual representation. In turn the moment measures uniquely define the $\mathcal{L}(\Xi(t))$ as
follows: Let $\wt\mu_1,\wt\mu_2,\dots$ be an exchangeable sequence of
$\mathcal{P}(\mathbb{I})$-valued random variables with marginal distributions given by the moment
measures determined by the dual $\mathcal{G}^2_t$  and \be{} \wt\Xi_m:=\frac{1}{m}\sum_{i=1}^m
\delta_{\wt\mu_i}.\ee
 Then by de Finetti's theorem \be{}
\mathcal{L}(\wt\Xi_m)\Rightarrow \mathcal{L}(\Xi(t))\quad \text{as }m\to\infty,\ee that is, the
time marginal laws of any solution are uniquely determined by this limit.

\medskip
\noindent (b) Since the class of function of the form $\mathcal{H}(\cdot,\mathcal{G})$ with
$\mathcal{G}\in (\mathcal{I})^{\otimes \N,*}$ is probability-measure-determining on
$\mathcal{P}(\mathcal{P}(\mathbb{I}))$ this implies that the time marginals of $\Xi(t)$  are
uniquely determined follows from (a). The result (b) then follows from the basic results on dual
martingale problems (see e.g. Theorem 7.9 and Proposition 7.10 in \cite{D10} or Chapter 4 of
\cite{EK2}).

\end{proof}

\bigskip

\subsubsection{Moment calculations}  The dual can be used to compute joint moments and covariance
structures. We illustrate with two simple examples.

\begin{example} \label{ex6}  Consider the case of $\mathbb{I}=\{1,2\}$, no mutation and $V_1(1)=1,V_1(2)=0$, $c=\gamma_2=s_2=0$ but $s_1,\gamma_1>0$.
In order to compute \be{}\lim_{t\to\infty} E_{\delta_{\mu_0}}\left(\int \mu(2)\Xi_t(d\mu)\right)\ee
we use the the dual started with $\mathcal{G}_0=(01)$. Then we have transitions due to selection
and coalescence. As a result \be{} \mathcal{G}_t= (01)^{\otimes_1 n(t)}\ee where $n(t)$ is a birth
and death process with linear birth rate $n\to n+1 $ at rate $s_1n$ and quadratic death rate
$\gamma_1n(n-1)/2$.  As a result $\{n(t)\}$ is ergodic with distribution measure $\{p_k:k\in \N\}$.
Then \be{}\label{lfirst} \lim_{t\to\infty} E_{\delta_{\mu_0}}\left(\int \mu(2)\Xi_t(d\mu)\right)=
E\left(\mu_0^{\otimes_1}(\mathcal{G}_{eq})\right)=\sum_{k=1^\infty} (\mu_0(2))^kp_k.\ee

\noi Next consider \be{}\label{lmixed} \lim_{t\to\infty} E_{\delta_{\mu_0}}\left(\int (\mu(1)\cdot
\mu(2))\,\Xi_t(d\mu)\right).\ee In this case we start the dual with
$\mathcal{G}_0=(10)\otimes_1(01)$. Again the number of ranks is given by a birth and death process
which will eventually reduce to $n(t)=1$ which due to coalescence implies that
$\mathcal{G}(\tau)=\emptyset$ for some finite random time $\tau$.  This implies fixation, that is,
the limit (\ref{lmixed}) is zero.  The corresponding fixation probabilities are then given by the
limiting first moments calculated in (\ref{lfirst}).
\end{example}

\begin{example}
With two types $\{1,2\}$ we take $\mathcal{G}^2_0= (10)_1^{\otimes k_1}\otimes_2 (10)_2^{\otimes
k_2}$ and we get
\[ E_{\Xi_0}\left(\int(\mu(1))^{k_1}\Xi_t(d\mu)\int(\mu(1))^{k_2}\Xi_t(d\mu)\right)=E[\mathcal{H}(\Xi_t,(10)_1^{\otimes k_1}\otimes_2 (10)_2^{\otimes k_2})]
=E_{\mathcal{G}^2_0}[\mathcal{H}(\Xi_0,\mathcal{G}^2_t)].\]
\end{example}

\subsubsection{Coalescent}

The coalescent plays a central role in the study of Moran, Fisher-Wright and Fleming-Viot processes
with neutral types. We now consider the analogous genealogical structure for two-level systems with
$s_1=s_2=0$ which is determined by  the level one and level two coalescence transitions in the
set-valued dual. The genealogy is described by a marked coalescent process analogous to the marked
coalescent process used for spatial processes (see for example \cite{GLW}).

 The state space of the two-level coalescent is the set of marked partitions $(\pi,\zeta(\pi)),
 \pi\in \Pi^{\texttt{I}},\zeta(\pi)\in \mathbb{N}^{|\pi|} $
 where $\Pi^{\texttt{I}}$ is the set of partitions of a countable set $\texttt{I}$ into subsets and for
 $\pi\in\Pi^{\texttt{I}}$, $|\pi|$ denotes the number of subsets in the partition.
 The marks $\{\zeta(i):i=1,\dots,|\pi|\}$
 represent the positions in $\N$ of the subsets and $|\zeta|=|\{k\in\mathbb{N}:\zeta(i)=k\text{ for some }i\in \{1,\dots,|\pi|\}\}|$.  For
 $i=1,\dots, |\zeta|$, let  $n_i(\pi)$ denote the number of subsets with $\zeta=i$ so that
 $\sum_{i=1}^{|\zeta(\pi)|}
 n_i(\pi)= |\pi|$.

A subset in the partition  can jump to a new unoccupied site at rate $c$ and level I coalescence of
two subsets occurs at rate $\gamma_1$ if they are at the same site. On the other hand all the
subsets at two occupied sites combine to form a single site with all these subsets at rate
$\gamma_2$.

Therefore given $(|\zeta(\pi)|, (n_1,\dots,n_{|\zeta(\pi)|}) = (k,(n_1,\dots,n_k))$ the possible
transitions are:
\begin{enumerate}
\item $(k,(n_1,\dots,n_k))\to (k+1,(n_1,\dots,n_i-1,\dots,n_k,1))$ at rate $cn_i 1_{n_i>1}$,
\item  $(k,(n_1,\dots,n_k))\to (k,(n_1,\dots,n_i-1,\dots,n_k))$ at rate $\gamma_1 n_i(n_i-1)$,
\item  $(k,(n_1,\dots,n_i,\dots,n_j,\dots,n_k))\to  (k-1,(n_1,\dots,n_i+n_j,\dots,n_k))$ at rate
$\gamma_2 k(k-1)$.
\end{enumerate}
\begin{proposition} (The multilevel coalescent)\\
(a) Consider the case with $s_1=s_2=c=0$, $\gamma_2>0$ and $\gamma_1>0$. Then the two level coalescent converges to $(1,(1))$.\\
(b) If $c>0$ and $\gamma_1>0,\; \gamma_2=0$, then the coalescent process started at
$(k_0,(n_1,\dots,n_{k_0}))$ with $k_0<\infty$, converges to $(\wt k,(1,\dots,1))$ for some random
$\wt k$.
\end{proposition}
\begin{proof} (a) Due to jumps of type 2 at rate $\gamma_2$  the process will eventually reach an element of the type
$(1,(n_1,\dots,n_i+n_j,\dots,n_1))$. The due to jumps of type 1 the element $(1,(1))$ is then
reached in a finite random time.\\
(b) If $\gamma_2 =0$, the number of occupied sites is nondecreasing and at each site level I
coalescence leads to a single element.

\end{proof}


\section[Multilevel population systems with two types]{Long-time behaviour of two type  multilevel population systems}

The class of two-level Fleming-Viot systems  obtained by Theorem \ref{T.1} with dual representation
given by Theorem \ref{T.3} describe a rich class of population systems.  The evolution of their
structure over different time scales depends on various parameters including mutation, levels I and
II selection, migration, and demographic stochasticity rates and lead to different classes of
behaviours.

In this section we  consider the simplest  case of a system with two types $\mathbb{I}=\{1,2\}$, no
mutation, migration rate $c$, levels I and II selection rates $s_1,s_2$ and levels I and II
resampling rates $\gamma_1,\gamma_2$. We first consider the deterministic case (i.e. infinite
population case at both levels, $\gamma_1=\gamma_2=0$) and then in the random cases $\gamma_1$
and/or $\gamma_2>0$.

\subsection{Nonlinear measure-valued dynamics with $\mathbf{\gamma_1=\gamma_2=0}$}

\begin{proposition}\label{P.4} Consider the two level system in which $\gamma_2= 0$ and $P(\Xi_0=\nu_0)=1,\; \nu_0\in\mathcal{P}(\mathcal{P}(\mathbb{I}))$.  Then the
$\mathcal{P}(\mathcal{P}(\mathbb{I}))$-valued process $\Xi_t=\nu_t$ is deterministic.
\end{proposition}
\begin{proof}
This follows immediately from the dual representation since the variance of $\rm{Var}(\int
h(\mu)\Xi_t(\mu))=0$ for any $h\in\mathbb{H}$. This can be shown to be zero since all $E[(\int
h(\mu)\Xi_t(\mu))^2]$ is obtained by starting the dual with
\[ h(\mu_1)\otimes_2 h(\mu_2)\]
and since no coalescence occurs the two descendant sets evolve independently.
\end{proof}

Now assume that the mutation, migration and genetic drift parameters are zero and

\be{}V_1(x)= 1_{D}(x),\; D\subset \mathbb{I},\ee and level II fitness function of the form: \be{}
V_2(\mu)= \mu(B)\text{
  with  } B\subset \mathbb{I}.\ee (This can be generalized to $V_2(\mu)=\mu^\otimes(B)$ with $B\subset (\mathbb{I})^K$ for some
$K\in\N$ but we consider here the case $K=1$ to keep things simple.)
\bigskip

 The transitions of the dual due to level I, respectively, level II,
are given by

\be{}\label{E.3.0}  1_C(x_{11})\to 1_{D\cap C}(x_{11})+1_{D^c}(x_{11})\otimes_1 1_C(x_{12})\quad
\text{at rate }s_1,\ee

\be{}\label{E.3.00}  1_C(x_{11})\to 1_B(x_{11})\otimes_11_{C}(x_{12})+1_{B^c}(x_{11})\otimes_2
1_C(x_{21})\quad \text{at rate }s_2.\ee

\bigskip
\begin{theorem}\label{T.5}
Assume that  $\mathbb{I}=\{1,2\}$ with fitness functions \be{}\label{E.3.1} V_1(1)=0,\;
V_1(2)=1,\ee and \be{}\label{E.3.2} V_2(\mu)=  \mu(1),\ee selection rates $s_1\geq 0,\; s_2\geq 2$
 and all other parameters equal to zero.
 Then as $t\to\infty$,\\
 (a) If $s_1>0,\; s_2=0$, and $\nu_0\ne \delta_{(1,0)}$, then $\nu_t \to (1-\nu_0((1,0))) {\delta_{(0,1)}}$,\\
 (b) If $s_1=0$, $s_2>0$, for some  $x_0\in (0,1]$, $\nu_0(\{\mu:\mu(1)\geq x_0\})>0$, then $\nu_t\to
 \delta_{(p^*,1-p^*)}$ with $p^*\geq \int \mu(1)\nu_0(d\mu)$.

\end{theorem}
\begin{proof}  We compute $\frac{d}{dt}E_{\nu_0}(\int\mu(1)\nu_t(d\mu))$ using the equivalent dual
form $\frac{d}{dt}E_{\mathcal{G}_0=\{1\}}(\mathcal{H}(\nu_0,\mathcal{G}_t))$. To compute the latter
 consider the first dual transition, either a level I selection jump (\ref{E.3.0})  or a level II
selection jump (\ref{E.3.00}). This yields  $1_1\to 1_1\otimes_1 1_1$ at rate $s_1$ and  $1_1\to
1_1\otimes_1 1_1 + 1_1\otimes_2 1_2$ at rate $s_2$ and therefore

 \noindent
\be{}\label{E.IN} \frac{d}{dt}E_t(\mu(1)) =s_2\text{Var}_{2,t}(\mu(1))
-s_1(E_t(\mu(1))-(E_t(\mu(1)))^2)\ee where
\[E_t(\mu(1))=\int\mu(1)\nu_t(d\mu),\;\text{Var}_{2,t}(\mu(1)) = \int \mu(1)^2\nu_t(d\mu)- (\int \mu(1)\nu_t(d\mu))^2\]

\noindent (a) If $s_2=0$, then (\ref{E.IN}) yields the ode

\be{} \frac{d}{dt}E_t(\mu(1))= s_1[(E_t(\mu(1)))^2-E_t(\mu(1))],\ee and the result follows
immediately.

\bigskip

\noindent (b) If $s_1=0$, then by (\ref{E.IN}) $E_t(\mu(1))$ is nondecreasing and strictly
increasing if $\text{Var}_{2,t}(\mu(1))>0$. Since $E(\mu(1))\leq 1$ and
$\text{Var}_{2,t}(\mu(1))>0$ unless $\nu =\delta_{(x,1-x)}$ for some $x\in [0,1]$,
$\lim_{t\to\infty} \text{Var}_{2,t}(\mu(1)) =0$ and $\nu_t\to \delta_{(p^*,1-p^*)}$ with
$p^*\geq\int \mu(1)\nu_0(d\mu)$.
\end{proof}

In the case $\delta_1=0$, $\delta_2=0$, the system is deterministic.  In order for level II
selection to play a role, diversity in the composition of clusters is required as first pointed out
by Maynard Smith (recall discussion in subsubsection 1.1.1). In particular consider the case where
$\nu_0\in\mathcal{P}(\mathcal{P}(\mathbb{I}))$, let $\widetilde{\nu}_0(dx):= \nu_0(\{\mu:\mu(1)\in
dx\})\in \mathcal{P}([0,1])$. The following is a version of a result of Luo \cite{S-13}.

\begin{theorem}\label{T.6}   Assume (\ref{E.3.1}) and (\ref{E.3.2}), $c=0$,  and $s_1>0$.  Assume that $\widetilde \nu_0(dx)$ has
a continuous density $\widetilde \nu_0(x)$.  Then \\
 (a) If there exists $x_0<1$ such that $\widetilde{\nu}_0([x_0,1])=0$,
 then $\widetilde \nu_t \to \delta_0,\;\; \nu_t\to \delta_{(0,1)}$.\\
(b)  If the density $\widetilde{\nu}_0$ is continuous and positive at $1$, then there exists a
critical $s_2^*$ such that if $s_2<s_2^*$, then $\nu_t \to {\delta_{(0,1)}}$ and if $s_2>s_2^*$,
then there exists an equilibrium distribution $\nu_{eq}$ with $\nu_{eq}(\{\mu:\mu(1)>0\})>0$ .

\end{theorem}
\begin{proof} This was established by S. Luo \cite{S-13} (see (A.5) in the case $\eta= U$) by solving an integro
partial differential equation for $\widetilde{\nu}_t(x)$, the density of the distribution of
$\mu_t(1)$. Luo's equation is given by \be{}\label{E.Luo} \frac{\partial}{\partial
t}\widetilde{\nu}(t,x)=\frac{\partial}{\partial
x}(\widetilde{\nu}(t,x)x(1-x))+\lambda\widetilde{\nu}(t,x)[x-\int_0^t \widetilde{\nu}(t,y)ydy]\ee
with solution \be{}\label{E.Luo2}
\widetilde{\nu}(t,x)=\widetilde{\nu}_0\left(\frac{xe^t}{1+x(e^t-1)}\right)e^{t-\lambda\int_0^th(s)ds}[1+x(e^t-1)]^{(\lambda-2)}\ee
where $\lambda=\frac{s_2}{s_1}$ and $h(t)=\int_0^1y\widetilde{\nu}(t,y)dy$.

\noindent (a) For any $\varepsilon >0$ there exists $t_0(\varepsilon)$ such that for $x\geq
\varepsilon$ and $t\geq t_0(\varepsilon)$, $\frac{xe^t}{1+x(e^t-1)}\geq x_0$ and therefore the
right side of (\ref{E.Luo2}) is 0.

 \noindent (b) In the case $\mu_0$ is uniform on $[0,1]$ and $\lambda\ne 1$, (\ref{E.Luo2}) becomes
\be{} \widetilde{\nu}(t,x)=\frac{(e^t-1)(\lambda
-1)}{e^{t(\lambda-1)}-1}[1+x(e^t-1)]^{(\lambda-2)}\ee where $\lambda = s_2/s_1$. When $\lambda <1$,
we have \be{} \int_0^1 x\widetilde{\nu}(t,dx)\to 0.\ee When  $\lambda >1$, we get
\be{}\label{E.lim} \widetilde{\nu}_{eq}(1)= \lim_{t\to\infty} \widetilde{\nu}(t,1) \to \lambda
-1,\ee and
 \be{} \int_0^1 x\widetilde{\nu}(t,dx)\to \frac{\lambda -1}{\lambda}.\ee

\end{proof}

\bigskip

From Theorem \ref{T.6}(a) it follows that the survival of type $1$  depends only on the initial
density $\widetilde{\nu}(0,x)$ near $x=1$ and (\ref{E.lim}) remains positive if the initial density
is positive at $x=1$. This means that level II selection can overcome level I selection with two
types only if there is a positive density of  sites (subpopulations) at time $t=0$ with arbitrarily
small (or zero) proportions of the individually more fit types.  We next consider the role of
randomness and demonstrate that level one positive genetic drift can lead to a phase transition
with the possibility of survival of inferior types for any initial $\nu_0$ provided that $\int
\mu(1)\nu_0(d\mu)\in (0,1)$.

\subsection{Phase transitions: the role of randomness}

In this subsection we again consider the case in which level two selection favours colonies that
include altruistic or cooperative types.  In the previous subsection level II selection exploited
the diversity in the initial distribution among demes.

\subsubsection{Level I randomness}\label{sss.rand1}

In this subsection we now consider the parameter regions for dominance of level I or level II
selective effects and the transition between these phases where level II selection acts on the
diversity in deme composition resulting from local genetic drift at each deme, that is, when
$\gamma_1>0$.

\begin{theorem}\label{T.7}
  Consider the case $\mathbb{I}=\{1,2\}$.
Assume that  , $c>0$, $\gamma_1 >0$,
 $\gamma_2=0$, $s_1\geq 0$, $s_2\geq 0$,
 with $\int
\mu(1)\nu_0(d\mu)\in (0,1)$.

\[ V_1(1)=0,\;V_1(2)=s_1.\] The migration rate is $c$ and the deme fitness is

\[V_2(\mu)=s_2\mu(1).\]
\medskip
\noi (a) Assume $m_{12}=m_{21}=0$ and $s_2=0,\, s_1 >0$.  Then $\int_{\mathcal{P}(\mathbb I)}
\mu(2)\nu(d\mu)\to 0$ with exponential decay rate.

\noi (b) Assume $m_{12}=m_{21}=0$ and $s_1=0,\, s_2 >0$.  Then $\int_{\mathcal{P}(\mathbb I)}
\mu(1)\nu(d\mu)\to 0$ with exponential decay rate.

 \noi(c) Assume $m_{12}=m_{21}=0$. Then for fixed $c>0,\;\gamma_1>0,\; s_1>0$, there is a critical value $s_2^*(c,\gamma,s_1)\in (0,\infty)$ such that level II selection dominates, that is
  then for $\varepsilon
>0$,
\[{  \nu_t(\{\mu:\mu(2)>\varepsilon\})\to 0,}\]
 if $s_2>s_2^*$
and level I selection dominates if $s_2<s_2^*$, that is,
 for $\varepsilon >0$,
\[{ \nu_t(\{\mu:\mu(1)>\varepsilon\})\to 0.}\]

\noi(d) Assume that $m_{12},m_{21}>0$.  The there exists a unique equilibrium and the system
converges to the equilibrium measure as $t\to\infty$.


\end{theorem}
\begin{proof} (a) and (b).  We use  the dual $\{\mathcal{G}_t\}$ with initial value
$\mathcal{G}_0= \{2\}\;(\text{  or  }\{1\})$ to compute the first moment of $E_{\Xi_0}\left[ \int
\mu(2)\Xi_t(d\mu)\right]= E_{\mathcal{G}_0}\left[ H(\Xi_0, \mathcal{G}_t)\right]$ where
$\Xi_0=\delta_{\mu_0}$.  Let $p=\mu_0(2)$.

We now identify two basic combinations of transitions that will  either increase or decrease this
expression.

\noi (1) Level I selection followed by migration before coalescence.
\medskip

\noi  Transitions in terms of indicator functions of the sets,
\begin{eqnarray*}
1_{2}(1)  && \rightarrow1_{2}(1)+1_{1}(1)\otimes_1 1_{2}(1)\quad\text{level I selection at rate }s_1\\
&& \rightarrow1_{2}(1)+1_{1}(1)\otimes_2 1_{2}(2)\;\;\text{migration}\end{eqnarray*} so that after
integration we obtain \be{} p  \rightarrow p+(1-p)p=2p-p^{2}=1-(1-p)^{2}.
\ee%
\noi Note that after the migration step, this change can no longer be reversed  by a coalescence at
site $1$. Noting that migration before coalescence occurs with probability $\frac{2c}{2c+\gamma_1}$
so that this combination occurs with effective rate $s_1\frac{2c}{2c+\gamma_1}$.

\noi Similarly these transition result in
\begin{eqnarray*}
1_{1}(1)  && \rightarrow0+1_{1}(1)\otimes_11_{1}(1)\quad\text{level I selection}\\
&& \rightarrow1_{1}(1)\otimes_21_{1}(2)\;\;\text{migration}\\
(1-p)  && \rightarrow(1-p)^{2}%
\end{eqnarray*}

\noi In the absence of level II selection this can be identified with a Crump-Mode-Jagers (CMJ)
branching process (see \cite{D10}(3.1.4), \cite{N} )  in which individuals are occupied sites and
during their lifetimes these individuals produce new offspring sites.  Since the death rate is zero
this is a supercritical branching process with exponential growth rate, the Malthusian parameter
$\alpha_1$ (see \cite{DG-14} subsubsection 8.3.4) which can be represented in terms of the stable
age distribution \cite{DG-14},(8.167). This proves (a).


\medskip
\bigskip

 \noi (2)  Level II selection. We consider the effect of a level II selection event.  followed by coalescence before migration.
\begin{eqnarray*}
1_{2}(1)  && \rightarrow1_{1}(1)\otimes_1 1_{2}(1)+1_{2}(1)\otimes_2 1_{2}(2)\quad\text{ level II
selection at rate }s_2\\&& \rightarrow 1_{1}(1)\otimes_21_{2}(2)+1_{2}(1)\otimes_21_{2}(2)=
1_2(2)\\&&\;\;\;\; \qquad\qquad\hspace{2cm}\text{migration before coalescence with probability
}\frac{2c}{2c+\gamma_1}
\\&& \rightarrow 0+1_{2}(1)\otimes_21_{2}(2)\;\;\;\;\text{coalescence before migration with
probability }\frac{\gamma_1}{2c+\gamma_1}  \end{eqnarray*} so that after integration we have

\be{} p\to p,\hspace{2cm} p\to p^2,\ee in the two cases.

Similarly in the second case (coalescence before migration) we have
\medskip%
\begin{eqnarray*}
1_{1}(1)  && \rightarrow1_{1}(1)\otimes_11_{1}(1)+1_{2}(1)\otimes_2 1_{1}(2)\quad\text{ level II selection}\\
&& \rightarrow1_{1}(1)+1_{2}(1)\otimes_21_{1}(2)\;\;\;\;\text{coalescence before migration with
probability }\frac{\gamma_1}{2c+\gamma_1},\end{eqnarray*}
 so that after integration we have
\be{}(1-p)   \rightarrow(1-p)+p(1-p)=1-p^{2}, \ee and this occurs with effective rate
$s_2\frac{\gamma_1}{2c+\gamma_1}$. \medskip
 In the absence of level I selection this again produces a supercritical CMJ branching process with Malthusian parameter
 $\alpha_2$ in which
individuals are occupied sites (where an occupied site is a site at which the set is not
$\mathbb{I}$).  This proves (b).

\bigskip

 \noindent (c) When both $s_1>0,\; s_2 >0$ there is a competition  between the two levels of
selection.  In the dual setting this corresponds to the the competition between two branching
mechanisms, that is, {\em competing CMJ branching processes}.

 \medskip

 We now focus on the interaction of the competing CMJ processes. To do this we first consider effect of a level I transition at a site where a level II transition
has already occurred. This leads to \be{} 1_2\to_{II}  1_2\otimes_2 1_2\to_{I} 1_2\otimes_2 1_2
+1_2\otimes_21_1\otimes_2 1_2,\ee that is, in terms of sets \[ \{2\}\supset
(\{2\}\otimes_2\{2\})\subset(\{2\}\otimes_2\{2\})\cup  (\{2\}\otimes_2\{1\}\otimes_2\{2\})\subset
\{2\}\]
 which after integration leads to \be{}  p\to_{II} p^2\to_{I}
p^2+p^2(1-p) \in [p, p^2]. \ee This means that a level I transition after a level II transition
partially reverses the first effect (but does not overshoot the reversal).  (The same happens if a
level II follows a level I transition.)  Therefore we can obtain a bound to the decrease  in the
mean of type 2 due to level II transitions by completely reversing them at the rate of level I
transitions.

First, assume that $s_1\frac{c}{c+\gamma_1}>s_{2}\frac{\gamma_1}{\gamma_1+c}$. We can construct a
birth and death process with births (new factors $1_1(\cdot)$ via level I selection) at rate
$s_1\frac{c}{c+\gamma_1}$.  We note that we can obtain a domination by letting the action of level
II selection to remove a $1_1(\cdot)$ (i.e. replacing it by $1_{2\cup 1}$ rather than something
intermediate), that is  this produces a death rate $s_2\frac{\gamma_1}{c+\gamma_1}$. Therefore  the
resulting birth and death process is supercritical and goes to infinity. Thus in this case,  if
$\int\mu(2)\nu_0(d\mu)>0$, then  $\int\mu(2)\nu_t(d\mu)\to 1$ and type 2 takes over in the limit in
the McKean-Vlasov system.

 Similarly, if $s_2\frac{c}{c+\gamma_1}>s_{1}\frac{\gamma_1}{\gamma_1+c}$ then we have a birth process with  $1_2$ factors produced due to level
 II
selection at rate $s_2\frac{\gamma_1}{c+\gamma_1}$ and a removal of these factors by level I
selection at rate $s_1\frac{\gamma_1}{c+\gamma_1}$.  Therefore if $\int\mu(1)\nu_0(d\mu)>0$, then
$\int\mu(1)\nu_t(d\mu)\to 1$ and type $1$ wins out and takes over in the McKean-Vlasov system.

\bigskip

\noi(d) This is a special case of Theorem \ref{T.ergodic2}.

\end{proof}

\bigskip

\begin{remark} We can also consider the {\em linear stability} of the fixed points $\delta_{(1,0)},\delta_{(0.1)}$
by computing
\[ \frac{d}{dt}E_{(\varepsilon,1-\varepsilon)}(\mu(1))|_{t=0} = {s_2 \gamma_1}-{s_1c} \]
 for small $\varepsilon$ (e.g. under the condition $s_2\leq c$ discussed below). This
is given by the first change due to a jump of the dual process started at $1_1$. This means that
level II selection prevails if \be{kcond}{s_2 \gamma_1}>{s_1c}.\ee Condition (\ref{kcond}) was
derived by Kimura \cite{K-83} using the properties of the equilibria of the solution of equation
\ref{Kim} for the density $U(t,x)$.
  Existence and
uniqueness  of the solution to (\ref{Kim}) was established by Shimakura \cite{Sm-85}), namely, for
any initial distribution there exists a unique solution.  Also see Shiga \cite{Sh-87} for a
generalization to the multitype case. A more complete description of the long-time behaviour is
proved in Ogura and Shimakura \cite{OS-87} Theorem 3(iii) (in the case $s_2\leq c$ so that
$s_1/\gamma_1 \leq 1$ if also $s_2>s_1c/\gamma_1$ which corresponds to the case in which the
expected number of dual factors at a site is bounded above by 2). They also obtain a related result
in Theorem 5 in the case in which the mutation rates satisfy $m_{12}>0,m_{21}>0$ or
$m_{12}>0,m_{21}=0$ or $m_{12}=0,m_{21}>0$. These results were obtained using ODE methods and
explicit solutions in terms of hypergeometric functions.

The ODE methods cannot be extended to populations having more than two types or systems with level
II coalescence. The objective of this paper is to develop the dual representation and the above
model of competing branching mechanisms to study the long-time behaviour of multitype-multilevel
systems. An example involving three types is given in Theorem \ref{T.cooperation}.
\end{remark}

\subsubsection{The case $\gamma_1>0,\;c=0$}

 Consider the case again with no mutation, level II  fitness
function $V_2(\mu)=\mu(2)$, $s_2>0$, $c=\gamma_2=0$ but with $\nu_0=\delta_{\mu_0}$, that is, no
initial diversity in the composition of demes but with $\mu_0(i)>0.\; i=1,2$.
 Then level II selection can dominate, that is, $\mu_t(2)\to 1$  only if
$\gamma_1>0$.   This means that group selection cannot be effective if $\gamma_1=0$ in the case
$\nu_0=\delta_{\mu_0}$.

In the case $c=0$, $\gamma_2=0$, $\gamma_1>0$,  fixation occurs within each deme so that eventually
the competition is between demes of type $1$ and demes of type $2$ and then $\mu_t(2)\to 1$  if
$s_2>0$.  To verify this we have \bea{} &&(10)\to (01)\otimes_1(10)+(10)\otimes_2(10)\quad\text{ at
rate  }s_2\\&&\qquad \to (10)\otimes_2(10)\quad\text{ at rate  }\gamma_1\nonumber\eea

\bea{}&& (10)\to (10)+(01)\otimes_1(10)\quad\text{ at rate }s_1\\&&\qquad \to (10) \quad\text{at
rate }\gamma_1.\nonumber\eea Recalling that the coalescence dominates level I selection here due to
the quadratic rate  then effectively the dual develops as $(10)^{\otimes_2 k(t)}$ and
$k(t)\to\infty$. Then if $\mu_0(1)<1$, $\nu_t=\delta_{\mu_t}$ and $\mu_t(1)\to 0$ as $t\to\infty$.

\subsubsection{Level II randomness $\mathbf{\gamma_2>0}$}

\beP{}   (a) If $\gamma_1=0$,   $\mu_0(2)<1.\;\nu_0$-a.s., and $s_1>0,\;\gamma_2>0,\; c>0$, then
\be{}E\left[\int\mu(2)\Xi_t(d\mu)\right]\to 0\text{   as }t\to\infty.\ee
\medskip

\noi (b) If $\gamma_1>0$ and $\gamma_2>0$, then fixation occurs, that is,
\be{}{}E\left(\int\mu(1)\Xi_t(d\mu)\cdot\int\mu(2)\Xi_t(d\mu)\right)\to 0,\quad \text{as   }
t\to\infty.\ee
\end{proposition}
\begin{proof}
(a) This follows by a simple dual calculation as follows: \be{} (01)\to (01)\otimes_1 (01) \text{
as rate }s_1,\ee

\be{} (01)\to (01)\otimes_1 (01) \to (01)\otimes_2 (01)\text{  at rate }c,\ee but then

\be{}(01)\otimes_2 (01)\to (01)\otimes_1(01)\text{  at rate }\gamma_2,\ee
 and \bea{} &&(01)\to (01)\otimes_1 (01)
+(10)\otimes_2(01)\text{  at rate } s_2\\&&\qquad\to (01)\otimes_1 (01) +(10)\otimes_1(01)=
(01)\text{ at rate } \gamma_2\nonumber\eea Recalling that the coalescence dominates level II
selection here due to the quadratic rate  then effectively the dual develops as $(01)^{\otimes_1
k(t)}$ and $k(t)\to\infty$. This implies that level II selection has no long term effect and the
level I selection leads to $\int (\mu(2))^{k(t)}\nu_0(d\mu)\to 0,\quad\text{as }t\to\infty$.
\medskip

\noi (b) This follows by a two-level  version  of the calculation in Example \ref{ex6}.
\end{proof}

\bigskip
\section{Multitype multilevel population systems}

In this section we consider the  evolution of more complex population systems  with  possibly many
types, multilevel structure and different combinations of the basic mechanisms.  In this case it is
no longer possible to use the methods of one dimensional nonlinear dynamics or  one dimensional
diffusions, that is, methods involving ordinary differential equations.  We will outline the
application of the dual representations in some examples in this class.

\subsection{ Equilibria and fixation probabilities}

An important application of the dual is to obtain results on ergodic properties of multilevel
systems.  We begin by reviewing the results for the case with only level I selection  in
\cite{DG-14}.

\begin{theorem}\label{T.ergodic} {\em Ergodicity of McKean-Vlasov systems with level I selection.}
Assume that $\gamma_2=0$ and  that at time $t=0$, all demes have the same individual distribution,
say $\mu_0$, that is $\nu_0=\delta_{\mu_0}$.  If the  mutation rates are positive on $\mathbb{I}$
and positive migration rate $c>0$, then the limiting empirical process is given by a deterministic
McKean-Vlasov dynamics $\{\nu_t\}_{t\geq 0}$ where $\nu_t=\mathcal{L}({\mu_t})$. Moreover,  as
$t\to\infty$, $\nu_t$ converges to a unique equilibrium $\nu_{eq}$.
\end{theorem}
\begin{proof} See \cite{DG-14} Theorem 12, Theorem 14 for the case $s_2=0$. The extension to
include the case $s_2>0$ follows along the same lines.
\end{proof}

 We now consider the extension of the ergodicity result to
the case in which $\gamma_2>0$.

\begin{theorem}\label{T.ergodic2} {\em Ergodicity of two level Fleming-Viot systems.} Assume  that
$\gamma_2 > 0$ and $s_2\geq 0$
 and  positive mutation rates on $\mathbb{I}$. Then the law of the two level
Fleming-Viot process converges to a unique equilibrium
$P_{eq}\in\mathcal{P}(\mathcal{P}(\mathcal{P}(\mathbb{I})))$.\\
\end{theorem}
\begin{proof} We adapt the proof of \cite{DG-14}. Due to level II coalescence the number of sites occupied by
the dual can be reduced to 1 and this event is recurrent.  Note that  the equilibrium mean measure
of a subset of $\mathbb{I}$ is given by the probability that the dual process starting at the
indicator of the subset hits the absorbing point $\mathbb{I}^{\N}$ but this occurs with positive
probability at each time the event occurs. Therefore an absorbing point is reached with probability
1.
\end{proof}

We now consider the case in which the mutation Markov chain has two or more ergodic classes.  In
this case the system can be non-ergodic. If $\gamma_1>0$, then eventually the population will be
concentrated on one of the ergodic classes and the problem is to compute the fixation
probabilities.

 \beT{}\label{T.fixation} {\em Fixation probabilities}
\begin{description}
\item [(a)]{\em  Single deme.} Assume that the initial configuration is iid $\mu_0$ and that the mutation chain has two or more ergodic classes $\mathbb{I}_1,\dots,\mathbb{I}_\ell$, $\gamma_1>0$,
and $s_1\geq 0$.  Then the population is ultimately fixed on one of the classes and the probability
that it is in class $k$ is given by the equilibrium measure of the dual chain started with $\mathcal{G}_0 = \mathbb{I}_k$ integrated with respect to the initial measure $\mu_0$
.
\item[(b)]  {\em Two level Fleming-Viot systems.} Assume that  $\gamma_1>0$, $\gamma_2 >0$, $V_1,V_2>0$ and that the mutation Markov chain has two or
more ergodic classes. Then there is ultimate fixation of a single ergodic class and  the law of the
two level system converges to a random mixture of pure equilibrium single class populations.

\end{description}

\end{theorem}
\begin{proof} (a) We again use the dual representation. To verify ultimate fixation take $\mathcal{G}_0 =\prod_{i=1}^{\ell^\prime} \mathbb{I}_i\subset \mathbb{I}^\ell$, $2\leq\ell^\prime\leq\ell$. Then due to the quadratic rate coalescence, with probability 1, $\mathcal{G}_\tau =\emptyset$ at some random time $\tau$ and this is a trap so that $\mathcal{G}_t =\emptyset$ for large $t$. To compute the probability that ergodic class  $k$ is chosen, start the dual with $\mathcal{G}_0 = \mathbb{I}_k$.  The dual process jumps are within class mutation jumps, selection jumps
and coalescence jumps. Due to coalescence the state $\mathcal{G} = \mathbb{I}_k$ is recurrent and
 therefore the dual is positive recurrent and converges to equilibrium as $t\to\infty$. The
fixation  probability (limiting probability that ergodic class $k$ is chosen) is then obtained by integrating with respect to the initial measure
$\mu_0^\otimes$.

\noi (b) We again use the dual process. Due to level II coalescence there will eventually be an
equilibrium distribution of the number of sites occupied by the dual and in fact the single site
situation occurs infinitely often.  Then each time there is non-zero probability that one of the
classes will be eliminated and the within class mutation will hit an absorbing point as in (a).

\end{proof}

\begin{example} Assume that $\mathbb{I}=\{1,2,3\}$, no mutation, $\gamma_1 >0,\;\gamma_2>0$ and
$\nu_0=\frac{1}{3}\delta_{\delta_{1}}+
\frac{1}{3}\delta_{\delta_{2}}+\frac{1}{3}\delta_{\delta_{3}}$. We use the set-valued dual
$\mathcal{G}_t$.  If  $\mathcal{G}_0=\{i\}_1\otimes_2\{j\}_2$,  then by level II coalescence
followed by level I coalescence,  $\mathcal{G}_\tau = \emptyset$ at some finite random time $\tau$
if $i\ne j$ and therefore only one type survives. We also note that the dual process
$\mathcal{G}_t$ starting from $\mathcal{G}_0=\{i\}_1$ is positive recurrent with $\{i\}_1$ as a
renewal point - this is due to the levels I and II quadratic rate  coalescence events  (in contrast
to the linear birth rates due to selection events).

Then ``level II fixation'' occurs, namely,
\[ \nu_t\to \delta_{\delta_{i}} \quad \text{with probability } p_i\] and the fixation probabilities $p_i$ are obtained
by integrating the equilibrium dual $\mathcal{G}_{\rm{eq}}(i)$ starting from
$\mathcal{G}_0=\{i\}_1$, that is,
\[ p_i= \nu_0^{\otimes}(\mathcal{G}_{\rm{eq}}(i)).\]

\end{example}

\subsection{Examples of multilevel effects}

The class of multitype multilevel population systems with mutation and selection is extremely rich
and can exhibit many complex dynamical behaviors.  We do not discuss this in detail but now given
an illustrative example of a simple effect of this type

\subsubsection
{A model of  cooperation.}\label{coop}
\medskip

\noindent Consider the 3 type case $\mathbb{I}=\{1,2,3\}$  with level I fitness function
\[ V_1(1)=v_1,\quad V_1(2)=V_1(3)=0,\]
and level II fitness function
\[  V_2(\mu)=v_2\mu(2)\mu(3),\]
with $v_1,v_2>0$.  This models a cooperative interaction of two types $2$ and $3$  that endow a
deme containing them  with a positive advantage.

\noindent We consider the emergence
 of demes having population distributions $(p_1,p_2,p_3)$ when
$\gamma_2=0$ satisfying \begin{equation}\label{threshold}  v_1p_1< v_2p_2p_3.\end{equation}


\begin{theorem}{}\label{T.cooperation}
Assume that  $c>0$,  $\gamma_1>0$, rare mutation from type 1 to types 2,3 occurs at rates
$m_{12}=m_{13}=\ve >0$  and $m_{23},m_{32}>0$, $m_{21}=m_{31}=0$.

(a)  If $v_2=0$, then for sufficiently large $v_1>0$ we can have long time survival of type $1$,
that is an equilibrium with positive mean proportion of type $1$.

(b) There exists a critical value $v_2^*$ at which emergence of demes having threshold values
satisfying (\ref{threshold}) occurs for $v_2>v_2^*$, type 1 becomes extinct and types $2,3$ go to
equilibrium.

\end{theorem}
\begin{proof} (a)  This has been proved for the case of two types in \cite{DG-14}, Corollary 10.1. The proof of (a) follows
along the same lines and elements of it are used in the proof of (b). We will outline the main
steps. We use the dual started with $\mathcal{G}_0=\{1\}\otimes \mathbb{I}^{\otimes\mathbb{N}}$.
The corresponding indicator function representation starts with  $(100)\otimes(111)^{\otimes
\mathbb{N}}$.    Then selection followed by migration leads to new summands. Mutation has the
effect of eliminating summands so that together we have a birth and death process. There are two
possible outcomes. First the $(100)$ can be a recurrent point and eventually will eventually change
to $(000)$ with probability 1.  The second possibility is that the birth and death process is
supercritical and in this case the type $1$ is positive for the invariant measure.

(b)  We consider $\lim_{t\to\infty}\int \mu(2\cup 3)\nu_t(d\mu)$ and $\lim_{t\to\infty}\int
\mu(1)\nu_t(d\mu)$.  Note that (using indicator functions of subsets of $\mathbb{I}$) the level II
selection transitions gives \bean{}&&(011)\to\\&& (010)\otimes_1(001)\otimes_1 (011)\\&&+
[(100)\otimes_1 (111)+(010)\otimes_1 (110)+(001)\otimes_1 (111)] \otimes_2 (011)\eean In order to
produce a permanent summand we then require either two migration or mutation transitions before
coalescence of the first two terms but this occurs with positive rate so that this can have an
important effect for large $v_2$.  This can then lead to

 \bean{}&&(011)\to\\&&
(010)\otimes_1(001)\otimes_1 (011)\\&&+ [(100)\otimes_1 (111)+(010)\otimes_1 (110)+(001)\otimes_1
(111)] \otimes_2 (011)\\&& \dots \to (011)+(100)\otimes_2 (011)\eean

Level I selection leads to

\[ (011)\to (011)\otimes_1 (011)\]

If we then have a level II selection

\bean{} &&(011)\to (011)\otimes_1 (011)\\&&\to (010)\otimes_1(001)\otimes_1 (011)\otimes_1
(011)\\&& + [(100)\otimes_1 (111)+(010)\otimes_1 (110)+(001)\otimes_1 (111)] \otimes_2
(011)\otimes_1(011)\\&&\to \dots (011)\otimes_1 (011)+(011)\otimes_1(100)\otimes_2(011)\eean
\bigskip

We then have competing branching mechanisms one whose rate is proportional to $v_1$ and one whose
rate is proportional to $v_2$.  We can construct birth and death processes where the deaths in the
level I process (birth rate prop. to $v_1$) are caused by mutation and level II selection. Deaths
in the level II process (birth rate prop  to $v_2$)are caused by level I selection.  Recall that
type 1 can only survive if the level birth and death process is supercritical (cf. DG, p. 743). We
again have a dichotomy involving the critical behaviour of two competing branching processes and
the result follows as in the proof of Theorem \ref{T.7}
\end{proof}

\subsubsection{Emergence of mutualistic types}\label{mutual}

There are many different mechanisms involving multilevel selection that can influence the overall
population structure. For example, multilevel selection can make possible the survival of a trait
that then leads to the emergence of another trait that has a mutually beneficial effect which then
gives the first trait higher (inclusive) individual fitness and the pair survives locally.

To illustrate this consider the following modification of the model of subsubsection \ref{coop}.
Let
\[ V_1(1)=v_1,\quad V_1(2,\mu)= v_{M}\cdot \mu(3), \;V_1(3,\mu)=v_{M}\cdot \mu(2),\] and level II fitness function
\[  V_2(\mu)=v_2(\mu(2)+\mu(3)).\]

\begin{theorem}{}\label{T.mutualism}

Assume that  $c\geq 0$,  $\gamma_1\geq 0$, rare mutation from type 1 to types 2,3 occurs at rates
$m_{12}=m_{13}=\ve \geq 0$  very small, and with $v_1,v_2, v_{M}\geq 0$. Denote by $x_i(t)$   the
expected proportion of type $i$ at time $t$.

\noindent (a)   Consider the deterministic case with $\gamma _1=m_{12}=m_{13}=0$, $v_1=1$ and
$v_2=0$ and for simplicity assume that $x_2(0)=x_3(0)>0$.  Then if \[ v_M<\frac{2}{1-x_1(0)},\]
then
\[ x_1(t)\to 1\] as $t\to\infty$.

\noindent (b)  If
\[ v_M >\frac{2}{1-x_1(0)},\]
then $ x_1(t)\to 0$ as $t\to\infty$.  In the case $ v_M=\frac{2}{1-x_1(0)}$ there is an unstable
equilibrium.

\noindent (c) Consider the two level system.  Assume conditions on $\nu_0$, namely on the density
at $1$ of the total mass of types $2$ and $3$ (as in Theorem \ref{T.6}) and  that $v_M>2$. Then for
sufficiently large $v_2$ type $1$ becomes extinct and the mutualistic types dominate. Moreover they
continue to dominate even if the level II selection ends, that is, $v_2(t)=1_{[0,T]}(t)\,v_2$ for
sufficiently large $T$.

\noindent
\noindent In this case we have a two stage process\\
\noindent Stage 1: Ignoring the mutualistic effect, survival  and equilibria of altruistic types $2,3$ by level II selection.\\
\noindent Stage 2: Takeover  by the  mutualistic pair $2,3$.\\

\end{theorem}

\begin{proof} (a,b)   Using the dual as above we can  obtain the following equations
by decomposing the first transitions of the dual into the different cases (recall (\ref{sdf})):
\[ \frac{dx_1(0)}{dt}= x_1(0)[(1-x_1(0))-\frac{v_m}{2}(1-x_1(0))^2].\]
The result follows by checking the sign of the derivative.

\noindent (c) By Theorem \ref{T.6}, in the two level system, under appropriate conditions on the
density of the total mass of types $2$ and $3$ at $1$, for sufficiently large $v_2$ there is an
equilibrium with mean mass of types $2,3$ greater than $\frac{2}{v_M}$, that is
$1-x_1(0)>\frac{2}{v_M}$, and therefore by part (b) type $1$ becomes extinct and the mutualistic
types dominate. Moreover they continue to dominate even if the level II selection ends.

\end{proof}

\begin{remark} Since the mutualistic pair  can then persist even in the absence of further level II selection in (c), this
illustrates the possible role of transient group selection in the long time genetic structure of a
population.  Without the role of level II selection, the simultaneous emergence of both types with
at least critical density would be an event of higher order rarity.
\end{remark}

\begin{remark}  We can also consider a random version of this.  In this case we assume that
$\nu_0=\delta_{(1,0,0)}$, $\gamma_1 >0$, $m_{12}=m_{13}=\varepsilon >0$.  Then again for
sufficiently large $v_M$ and $v_2$ type 1 dies out and the mutualistic types take over. The dual
analysis involves four classes of selection birth events corresponding to the level I fitness of
type 1, the mutualist fitnesses of types 2,3, the level II fitness of sites containing types 2 and
3 and level I coalescence.  This will be carried out elsewhere.
\end{remark}

\subsubsection{A three level system}

Consider a system with state space $\mathcal{P}(\mathcal{P}(\mathcal{P}(\mathbb{I})))$. For
example, this could model a system of competing regions in which regions contain competing towns
and each town contains a population with type space $\mathbb{I}$.  The relative fitness $V_3(\nu)$,
of a region (level III fitness) is assumed to depend on the distribution $\nu$ of the
characteristics of the towns it contains so that $V_3:(\mathcal{P}(\mathcal{P}(\mathbb{I})))\to
[0,1]$

\be{} V_3(\nu)=\prod_{k=1}^K \left[\int h_{k}(\mu_{k})\nu(d\mu_k)\right] \ee where $h_k$ is of the
form (\ref{alg2}), and convex combinations of functions of this form.  We also allow individuals to
migrate to a different region and even entire towns to move to a different region.  Then the
combined effect of three levels of selection can lead to complex behaviour which can be analyzed
using the set-valued dual with values in $((\mathcal{I})^
{\otimes\mathbb{N}})^{\otimes\mathbb{N}}$.




\subsection{The study of more complex multilevel interactions - set-valued Monte Carlo approximation}

For  more complex multilevel interactions it is natural to consider simulations since closed form
solutions cannot be expected. However the numerical simulation of systems of Fleming-Viot processes
involves the solutions of systems of nonlinear stochastic partial differential equations with
degenerate boundary behaviour.  The numerical simulation of such  systems of stochastic
differential equations  is difficult.   On the other hand the dual formulation involves only the
simulation of continuous time Markov chains with discrete states which are easy to simulate by
Monte Carlo methods. In particular one can compute means and covariances directly by simulating the
dual process and determining the empirical distribution of the outcomes, namely, of the absorption
probabilities (for equilibria) or equilibrium probabilities for fixation probabilities.

\subsection{ Extensions}

The study of multilevel evolutionary systems arises in evolutionary biology, ecology, virology,
sociology, economics, etc.  These give rise to a wide range of mechanisms and interactions. Some of
these can be modeled by extensions of the models and methods described above.  For example, the
models and the set-valued dual representations can be extended to systems with countably many
types, recombination and horizontal gene transfer, higher level interactions, random environments,
multiple species, and measure-valued processes with jumps leading to duals with multiple
coalescence.


\end{document}